\documentclass[a4paper]{article}

\usepackage{amsmath,amssymb,amsthm,centernot,indentfirst}
\usepackage[dvipdfmx]{hyperref}

\theoremstyle{plain} 
	\newtheorem{Thm}{Theorem}[section]
	\newtheorem{Cor}[Thm]{Corollary}
	\newtheorem{Lem}[Thm]{Lemma}
	
	\newtheorem{Prop}[Thm]{Proposition}
\theoremstyle{definition}
        \newtheorem{Def}[Thm]{Definition}
        
	\newtheorem{Rem}[Thm]{Remark}

\numberwithin{equation}{section}

\hypersetup{
setpagesize=false,
 bookmarksnumbered=true,%
 bookmarksopen=true,%
 colorlinks=true,%
 linkcolor=black,
 citecolor=black,
}

\newcommand{\R}{\mathbb{R}}
\newcommand{\Sbb}{\mathbb{S}}
\newcommand{\bary}{\mathrm{bar}}
\newcommand{\inn}[2]{\left\langle#1, #2\right\rangle}
\newcommand{\Ent}{\mathrm{Ent}}
\newcommand{\Lc}{\mathcal{L}}

\newcommand{\notll}{\centernot\ll}
\newcommand{\Int}{\mathrm{int}}
\newcommand{\argmin}{\mathop\mathrm{arg~min}\limits}

\begin{document}

\title{Symmetrized Talagrand Inequalities on Euclidean Spaces}
\author{Hiroshi Tsuji\footnote{Department of Mathematics, Osaka University, Osaka 560-0043, Japan (u302167i@ecs.osaka-u.ac.jp)}}
\date{}
\maketitle

\begin{abstract}
In this paper, we study the symmetrized Talagrand inequality that was proved by Fathi and has a connection with the Blaschke-Santal\'{o} inequality in convex geometry. 
As corollaries of our results, we have several refined functional inequalities under some conditions.
We also give an alternative proof of Fathi's symmetrized Talagrand inequality on the real line and some applications.   
\end{abstract}

\tableofcontents

\section{Introduction}

The Talagrand inequality, which is also called the Talagrand transportation inequality, is as follows: 
If $m = e^{-V}\mathcal{L}^n$ is a probability measure on $\R^n$ with $\nabla^2V\geq \kappa$ for some $\kappa>0$, 
and $\mu \in P_2(\R^n)$, then $W_2^2(\mu,m)\leq 2\Ent_m(\mu)/\kappa$ holds, 
where $\mathcal{L}^n$ is the Lebesgue measure on $\R^n$, $P_2(\R^n)$ is the set of all probability measures on $\R^n$ with finite second moment, 
$W_2$ is the Wasserstein distance, and $\Ent_{m}$ is the relative entropy (or the Kullback-–Leibler distance) with respect to $m$.
More generally, it is known that the Talagrand inequality holds on metric measure spaces with similar conditions above, and 
there are many studies on refinements of the Talagrand inequality and relations with logarithmic Sobolev inequalities and Poincar\'{e} inequalities (\cite{V2}). 
This paper is motivated by Fathi's following result.  

\begin{Thm}[\cite{F}]\label{STI}
Let $\mu, \nu \in P_2(\R^n)$. 
\begin{description}
\item{(1)} Let $m = e^{-V}\mathcal{L}^n$ be a probability measure on $\R^n$ such that $V \in C^{\infty}(\R^n)$ is even and  $\kappa$-convex for some $\kappa>0$. 
If $\nu$ is symmetric (i.e. its density with respect to $\Lc^n$ is even), then it holds that 
\begin{align*}
\frac{1}{2}W_2^2(\mu, \nu)\leq \frac{1}{\kappa}(\Ent_{m}(\mu)+\Ent_{m}(\nu)).
\end{align*}
\item{(2)} Let $m = \gamma_n$ be the $n$-dimensional standard Gaussian measure. If $\bary(\nu):=\int_{\R^n}x ~d\nu(x)=0$, then it holds that 
\begin{align*}
\frac{1}{2}W_2^2(\mu, \nu)\leq \Ent_{\gamma_n}(\mu)+\Ent_{\gamma_n}(\nu).
\end{align*}
Moreover, the equality holds in (2) if and only if there exist some positive definite symmetric matrix $A\in \R^{n\times n}$ and some $a\in \R$ such that $\mu$ is the Gaussian measure 
whose center is $a$ and covariance matrix is $A$, and $\nu$ is the Gaussian measure whose center is $0$ and covariance matrix is $A^{-1}$. 
\end{description}
\end{Thm}

Notice that Theorem \ref{STI}(1) does not include (2).
When $m=\nu$ in (1), we recover the classical Talagrand inequality, and hence 
Theorem \ref{STI} is a refinement of the classical Talagrand inequality. 
Using Fathi's paper as reference, we call this type inequality a {\it symmetrized Talagrand inequality}. 
Fathi proved the symmetrized Talagrand inequality by using optimal transport theory and convex geometry. 
Moreover, he pointed out that the symmetrized Talagrand inequality for Gaussian measures is related to the functional Blaschke-Santal\'{o} inequality, 
which is well-known and important in convex geometry. 

We consider refinements and extensions of the symmetrized Talagrand inequality in this paper. 
In general, the Talagrand inequality follows from the convexity of the relative entropy (Theorem \ref{TI}), and  
our idea is to strengthen the convexity under certain conditions as follows.   
\begin{Thm}\label{MT1}
Let $\mu_i = e^{-V_i}\mathcal{L}^n \in P_2(\R^n)$ $(i = 0,1)$ be probability measures with 
$V_i \in C^{\infty}(\R^n)$ $(i=0, 1)$.
We assume that $\mu_0$ and $\mu_1$ satisfy the following three conditions: 
\begin{description}
\item{(i)} $\bary(\mu_0) = \bary(\mu_1)$.
\item{(ii)} $\nabla^2V_0 \leq \kappa_0$ and $\nabla^2 V_1 \geq  \kappa_1$ for some $\kappa_i>0$ $(i=0,1)$.
\item{(iii)} $\mu_0$ satisfies the Poincar\'{e} inequality with a constant $C_{\mu_0}>0$ in the sense that, 
for any $f\in H^1(\R^n,\mu_0)$ with $\int_{\R^n}f ~d\mu_0 = 0$, it holds that 
\[
C_{\mu_0}\int_{\R^n}f^2~d\mu_0 \leq \int_{\R^n}\|\nabla f\|_2^2~d\mu_0.
\] 
\end{description}
Then $\Ent$, which is the relative entropy with respect to $\Lc^n$, is $C_{\mu_0}\min\left\{1,\kappa_1/\kappa_0\right\}$-convex along $(\mu_t)_{t\in [0, 1]}$, where 
$(\mu_t)_{t\in [0, 1]}$ is the geodesic from $\mu_0$ to $\mu_1$ in $(P_2(\R^n), W_2)$, and 
$H^1(\R^n,\mu_0)$ is the Sobolev space with respect to the probability measure $\mu_0$.  
\end{Thm}

\begin{Thm}\label{MT2}
Let $\mu_i = e^{-V_i}\mathcal{L}^n \in P_2(\R^n)$ $(i = 0,1)$ be probability measures with 
$V_i \in C^{\infty}(\R^n)$ $(i=0, 1)$, and let $m = e^{-V}\mathcal{L}^n$ be a probability measure on $\R^n$ such that $V \in C^{\infty}(\R^n)$ satisfies $\kappa\leq V\leq \kappa'$ for some $\kappa, \kappa'>0$. 
We assume that $\mu_0$ and $\mu_1$ satisfy the following two conditions: 
\begin{description}
\item{(i)} $\bary(\mu_0) = \bary(\mu_1)$.
\item{(ii)} $\nabla^2V_0 \geq \kappa_0$ and $\nabla^2 V_1 \geq  \kappa_1$ for some $\kappa_i>0$ $(i=0,1)$.
\end{description} 
Then $\Ent_{m}$ is $\kappa(1 + \min\{\kappa_0/\kappa', \kappa_1/\kappa'\})$-convex along  $(\mu_t)_{t\in [0, 1]}$, 
where $(\mu_t)_{t\in [0, 1]}$ is the  generalized geodesic from $\mu_0$ to $\mu_1$ with the base $m$ in $(P_2(\R^n), W_2)$.
\end{Thm}

In general, a function $f$ on a metric space $(X,d)$ is said to be $\kappa$-convex for some $\kappa\in\R$ 
if 
\[
f((1-t)x +ty)\leq (1-t)f(x) + tf(y) -\kappa t(1-t)d(x,y)^2/2 
\] 
for any $x, y\in X$ and any $t\in[0,1]$.  
In particular, when $X=\R^n$ and $f\in C^{\infty}(\R^n)$, $f$ is $\kappa$-convex for some $\kappa \in \R$ if and only if $\nabla^2f\geq \kappa$. 
A generalized geodesic in Theorem \ref{MT2} is defined in Subsection \ref{GG}. 
Theorems \ref{MT1} and \ref{MT2} yield the following symmetrized Talagrand inequalities, respectively. 

\begin{Cor}\label{MC1}
Let $\mu_i$ $(i = 0,1)$ be as in Theorem \ref{MT1}, and 
$m = e^{-V}\mathcal{L}^n$ be a probability measure on $\R^n$ such that $V \in C^{\infty}(\R^n)$ is $\kappa$-convex for some $\kappa>0$. 
Then it holds that 
\[
\frac{1}{2}W_2^2(\mu_0, \mu_1)\leq \frac{2}{\kappa + C_{\mu_0}\min\left\{1, \frac{\kappa_1}{\kappa_0}\right\}}(\Ent_m(\mu_0)+\Ent_m(\mu_1)).
\]
In particular, if there exists $\kappa_2>0$ such that $\kappa_2 \leq \nabla^2 V_0$ $(\leq\kappa_0)$, then 
\[
\frac{1}{2}W_2^2(\mu_0, \mu_1)\leq \frac{2}{\kappa + \kappa_2\min\left\{1, \frac{\kappa_1}{\kappa_0}\right\}}(\Ent_m(\mu_0)+\Ent_m(\mu_1)).
\]
Therefore if $\kappa, \kappa_i$ $(i=0,1,2)$ satisfy $\kappa\leq\kappa_2 \leq \kappa_0\leq\kappa_1$, then 
\[
\frac{1}{2}W_2^2(\mu_0, \mu_1)\leq \frac{1}{\kappa}(\Ent_m(\mu_0)+\Ent_m(\mu_1)).
\]
\end{Cor}

\begin{Cor}\label{MC2}
Let $\mu_i$ $(i = 0,1)$ and $m$ be as in Theorem \ref{MT2}.
Then 
\[
\frac{1}{2}W_2^2(\mu_0, \mu_1)\leq \frac{1}{\kappa}\cdot\frac{2}{1+\min\left\{\frac{\kappa_0}{\kappa'},\frac{\kappa_1}{\kappa'}\right\}}(\Ent_{m}(\mu_0)+\Ent_m(\mu_1)).
\]
In particular, if $\kappa_0$ and $\kappa_1$ satisfy $\kappa_0, \kappa_1 \geq \kappa'$, then 
\[
\frac{1}{2}W_2^2(\mu_0, \mu_1)\leq \frac{1}{\kappa}(\Ent_{m}(\mu_0)+\Ent_m(\mu_1)).
\]
\end{Cor}

In Corollaries \ref{MC1} and \ref{MC2}, we do not assume that probability measures $m, \mu_0$ and $\mu_1$ 
are symmetric.
In this sense, our symmetrized Talagrand inequalities are more general than Fathi's inequalities. 
We find the assumption (i) in Corollaries \ref{MC1} and \ref{MC2} to be naturally derived from Theorem \ref{STI}(2) (see Section 3). 

Besides the results above, we discuss an alternative proof of Theorem \ref{STI}. In particular, we can prove an extension of Theorem \ref{STI}(1) and apply it to well-known inequalities in convex geometry.

The present paper is organized as follows. In the next section, we introduce some fundamental notions from optimal transport 
theory and functional inequalities including Talagrand inequalities. 
In Section 3, we give another form of the symmetrized Talagrand inequality by a self-improvement of Fathi's result  (Theorem \ref{STI}(2)). 
The barycenter of a probability measure plays an important role in this section. 
In Section 4, we prove Caffarelli's contraction theorem under relaxed conditions compared with classical Caffarelli's result, 
which is the key theorem in order to prove the main theorems. 
In  Section 5, we prove the main theorems and apply them to prove the corresponding HWI inequalities, logarithmic Sobolev inequalities and Poincar\'{e} inequalities. 
Moreover, we give an alternative proof of Theorem \ref{STI} and an extension on the real line in Subsection 5.3. 
In the final section, we describe some applications of the result in the previous subsection to convex geometry, in particular, to the concentration of measures and the Blaschke-Santal\'{o} inequality, 
those generalize classical ones. 
%%%%%%%%%%%%%%%%%%%%%%%%%%%%%%%%%%%%%%%%%%%%%%%%%%%%%%%%%%%%%%%%%%%%%%%%%%%%%%%%%%%%%%%%%%%%%%%%%%%%%%%%%%%%%%%%%%%%%%%%%%%%%%%%%%%%%%%%%%%%%%%%%%%%%%%%%%%%%%%%%%%%%%%%%%%%%%%%%%%%%%%%%%%%%%%%
\section*{{\bf Acknowledgments}}
The author would like to thank his supervisor, Prof. Shin-ichi Ohta %, who guided him to write this paper, 
for warm encouragement, reviewing a preliminary version of this paper, and useful comments.
He also thanks his colleagues for conversations on this subject.

%%%%%%%%%%%%%%%%%%%%%%%%%%%%%%%%%%%%%%%%%%%%%%%%%%%%%%%%%%%%%%%%%%%%%%%%%%%%%%%%%%%%%%%%%%%%%%%%%%%%%%%%%%%%%%%%%%%%%%%%%%%%%%%%%%%%%%%%%%%%%%%%%%%%%%%%%%%%%%%%%%%%%%%%%%%%%%%%%%%%%%%%%%%%%%%%%
%%%%%%%%%%%%%%%%%%%%%%%%%%%%%%%%%%%%%%%%%%%%%%%%%%%%%%%%%%%%%%%%%%%%%%%%%%%%%%%%%%%%%%%%%%%%%%%%%%%%%%%%%%%%%%%%%%%%%%%%%%%%%%%%%%%%%%%%%%%%%%%%%%%%%%%%%%%%%%%%%%%%%%%%%%%%%%%%%%%%%%%%%%%%%%%%%
%%%%%%%%%%%%%%%%%%%%%%%%%%%%%%%%%%%%%%%%%%%%%%%%%%%%%%%%%%%%%%%%%%%%%%%%%%%%%%%%%%%%%%%%%%%%%%%%%%%%%%%%%%%%%%%%%%%%%%%%%%%%%%%%%%%%%%%%%%%%%%%%%%%%%%%%%%%%%%%%%%%%%%%%%%%%%%%%%%%%%%%%%%%%%%%%%
%%%%%%%%%%%%%%%%%%%%%%%%%%%%%%%%%%%%%%%%%%%%%%%%%%%%%%%%%%%%%%%%%%%%%%%%%%%%%%%%%%%%%%%%%%%%%%%%%%%%%%%%%%%%%%%%%%%%%%%%%%%%%%%%%%%%%%%%%%%%%%%%%%%%%%%%%%%%%%%%%%%%%%%%%%%%%%%%%%%%%%%%%%%%%%%%%
\section{Preliminaries}

In this section, we describe several well-known facts from optimal transport theory and various functional inequalities needed in subsequent sections. 
For general theory and discussions in detail, for example, see \cite{V1} and \cite{V2}. 
We give proofs of some results for applications described in Section 5. 

%%%%%%%%%%%%%%%%%%%%%%%%%%%%%%%%%%%%%%%%%%%%%%%%%%%%%%%%%%%%%%%%%%%%%%%%%%%%%%%%%%%%%%%%%%%%%%%%%%%%%%%%%%%%%%%%%%%%%%%%%%%%%%%%%%%%%%%%%%%%%%%%%%%%%%%%%%%%%%%%%%%%%%%%%%%%%%%%%%%%%%%%%%%%%%%%%
\subsection{Optimal transport theory}

Let us denote by $\inn{\cdot}{\cdot}$ the standard Euclidean inner product on $\R^n$, and denote by $\|\cdot\|_2$ the standard Euclidean norm on $\R^n$. 
We also define $P_2(\R^n)$ as the set of all Borel probability measures on $\R^n$ with $\int_{\R^n}\|x\|_2^2 ~d\mu(x) < \infty$. 
For any two probability measures $\mu$ and $\nu$ in $P_2(\R^n)$, a probability measure $\pi$ on $\R^n\times\R^n$ is said to be a {\it coupling of $\mu$ and $\nu$} if 
for any Borel subset $A\subset\R^n$, $\pi(A\times \R^n) = \mu(A)$ and $\pi(\R^n\times A) = \nu(A)$ hold, and we denote by $\Pi(\mu,\nu)$ the set of all couplings of $\mu$ and $\nu$. 
  
For any $\mu,\nu \in P_2(\R^n)$, the {\it $(L^2$-$)$Wasserstein distance} is defined as follows: 
\[
W_2(\mu, \nu) := \left(\inf_{\pi \in \Pi(\mu, \nu)}\int_{\R^n\times \R^n}\|x - y\|_2^2 ~d \pi(x, y) \right)^{\frac{1}{2}}.
\]
In fact, $W_2$ is a distance function on $P_2(\R^n)$, and in addition, the metric space $(P_2(\R^n), W_2)$ is a geodesic space. 
This space is called the $(L^2$-){\it Wasserstein space}. 
In the following, we describe some well-known and important facts from optimal transport theory.  

The next result is known as the Kantorovich duality. 
\begin{Thm}[\cite{V1}]\label{KD}
For any $\mu, \nu \in P_2(\R^n)$, it holds that 
\begin{align*}
W_2^2(\mu, \nu) = &\sup\Bigg\{\int_{\R^n}f ~d\mu + \int_{\R^n}g ~d\nu ~
\Bigg{|}~ f\in L^1(\R^n,\mu), g\in L^1(\R^n,\nu), \\
&\hspace{7cm} f(x) + g(y) \leq \|x - y\|_2^2 ,~\forall x,y\in \R^n \Bigg\}.
\end{align*}
\end{Thm}

\begin{Rem}\label{rem}
Although we consider the supremum in the right hand side above for functions $f$ and $g$ which are in $L^1(\R^n,\mu)$ and $L^1(\R^n,\nu)$, respectively, 
we can restrict these functions to the class $C_b(\R^n)$ consisting of all bounded continuous functions on $\R^n$ (see \cite[Theorem 1.3]{V1}). 
\end{Rem}

It follows from elementary probability theory that there exists some coupling attaining the infimum of the Wasserstein distance in the definition, which is called the {\it optimal transport plan} or {\it optimal coupling} (see \cite{V1}). 
In order to explain a property of optimal transport plans, we need the following notion. 
\begin{Def}
$\Gamma\subset \R^n\times\R^n$ is said to be {\it cyclically monotone} if for any integer $k\geq1$ and any $(x_1,y_1),(x_2,y_2),\dots,(x_k,y_k)\in\Gamma$, it holds that 
\[
\sum_{i=1}^k\|x_i-y_i\|_2^2 \leq \sum_{i=1}^k \|x_i-y_{i-1}\|_2^2,
\]
where we put $y_0:=y_k$.
\end{Def}

\begin{Thm}[\cite{V1}]\label{CM}
Let $\mu,\nu\in P_2(\R^n)$, and let $\pi\in \Pi(\mu,\nu)$ be an optimal transport plan of $\mu$ and $\nu$.  
Then $\mathrm{supp}(\pi)\subset \R^n\times\R^n$ is cyclically monotone. 
\end{Thm}

Moreover, when a probability measure $\mu$ is absolutely continuous with respect to the Lebesgue measure, then an optimal transport plan of $\mu$ and another probability measure $\nu$  
can be induced from some map. 
The next theorem is known as Brenier's theorem. 
\begin{Thm}[\cite{V1}]\label{BT}
Let $\mu, \nu\in P_2(\R^n)$ be such that $\mu$ is absolutely continuous with respect to the Lebesgue measure. 
Then there exists some map $T : \R^n\to\R^n$ satisfying $T_{\#}\mu =\nu$ and 
\[
W_2^2(\mu, \nu) = \int_{\R^n}\|T(x) - x\|_2^2 ~d\mu(x),
\]
where $T_{\#}\mu$ is the image measure of $\mu$ by $T$. 
Moreover, $T$ is unique up to a difference on a null measure set and coincides with the gradient of some convex function on $\R^n$. 
In addition, $T$ is a locally Lipschitz map. 
\end{Thm}
The map $T$ in Theorem \ref{BT} is called the {\it optimal transport map} from $\mu$ to $\nu$, and the convex function generating $T$ is called the {\it Kantorovich potential}. 
Using optimal transport maps, we can concretely represent geodesics between two probability measures in $(P_2(\R^n),W_2)$. 
Indeed, when $\mu\in P_2(\R^n)$ is absolutely continuous with respect to the Lebesgue measure, $\nu\in P_2(\R^n)$, and $T$ is the optimal transport map from $\mu$ to $\nu$, 
then $\mu_t:=((1-t)id+tT)_{\#}\mu$ $(t\in[0,1])$ is the geodesic from $\mu$ to $\nu$ in $(P_2(\R^n),W_2)$, 
where $id$ is the identity map on $\R^n$. 

The equation in the next theorem is called the {\it Monge-Amp\`{e}re equation}. 
\begin{Thm}[\cite{V1}]\label{MA}
Let $\mu,\nu\in P_2(\R^n)$ be absolutely continuous with respect to the Lebesgue measure, and denote their densities by $f$ and $g$, respectively. 
Let $T:\R^n \to \R^n$ be the optimal transport map from $\mu$ to $\nu$. 
Then it holds that $f(x) = g(T(x)) \det \nabla T(x)$ for $\mu$-a.e. $x\in \R^n$. 
\end{Thm}

Finally, we describe optimal transport theory on the real line which is needed in Subsection 5.3. 
\begin{Thm}[\cite{V1}]\label{1dim}
Let $\mu$ and $\nu$ be two probability measures on $\R$ such that $\mu$ is absolutely continuous with respect to the Lebesgue measure. 
Define the functions $F$ and $G$ on $\R$ as $F(x):=\mu((-\infty,x])$ and $G(x):=\nu((-\infty,x])$, respectively, and 
the function $G^{-1}$ on $(0,1)$ as $G^{-1}(x):=\inf\{y\in\R ~|~G(y)>x\}$. 
Then $F$ is a monotone increasing function, and $G^{-1}\circ F$ is an optimal transport map from $\mu$ to $\nu$. 
In particular, when $\mu, \nu$ and $m$ in $P_2(\R)$ are absolutely continuous with respect to the Lebesgue measure, $m$ satisfies $\mathrm{supp}(m)=\R$, 
and $T_1$ and $T_2$ are the optimal transport maps from $m$ to $\mu$ and $\nu$, respectively, 
then $T_1$ and $T_2$ are locally Lipschitz and strictly monotone increasing functions satisfying $W_2^2(\mu,\nu) =\int_{\R}(T_1-T_2)^2 ~dm$. 
\end{Thm}
%%%%%%%%%%%%%%%%%%%%%%%%%%%%%%%%%%%%%%%%%%%%%%%%%%%%%%%%%%%%%%%%%%%%%%%%%%%%%%%%%%%%%%%%%%%%%%%%%%%%%%%%%%%%%%%%%%%%%%%%%%%%%%%%%%%%%%%%%%%%%%%%%%%%%%%%%%%%%%%%%%%%%%%%%%%%%%%%%%%%%%%%%%%%%%%%%%
\subsection{Various functional inequalities}
 
Firstly, we define the Shannon entropy and relative entropy. 
Let $\mu \in P_2(\R^n)$. 
Then the {\it Shannon entropy of $\mu$} is defined by 
\begin{align*}
\Ent(\mu) := 
\begin{cases}
\int_{\R^n}\rho(x)\log \rho(x) ~dx  &\quad \text{if} \quad \mu= \rho \Lc^n, \\
\infty &\quad \text{if} \quad \mu\notll\Lc^n.
\end{cases}
\end{align*}
Similarly, for a probability measure $m$ on $\R^n$, the {\it relative entropy of $\mu$ with respect to $m$} is defined by 
\begin{align*}
\Ent_m(\mu) := 
\begin{cases}
\int_{\R^n}\rho\log \rho ~dm &\quad \text{if} \quad \mu= \rho m, \\
\infty &\quad \text{if} \quad \mu\notll m.
\end{cases}
\end{align*}
Such a probability measure $m$ is called the {\it reference measure}. 
By Jensen's inequality, the relative entropy is nonnegative. 
Moreover, $\Ent_{m}(\mu) = 0$ holds if and only if $m=\mu $. 

The next theorem asserts that the Shannon entropy and relative entropy are convex on $(P_2(\R^n),W_2)$. 
\begin{Thm}[\cite{V2}]\label{convexity}
\begin{description}
\item{(1)} Along any geodesic $(\mu_t)_{t \in[0,1]}$ in $(P_2(\R^n), W_2)$, we have 
\[
\Ent (\mu_t) \leq (1-t)\Ent (\mu_0) + t\Ent (\mu_1)
\]
for all $t\in[0,1]$.
\item{(2)} Let $m = e^{-V}\mathcal{L}^n$ be a probability measure on $\R^n$ such that $V \in C^{\infty}(\R^n)$ is  $\kappa$-convex for some $\kappa>0$. 
Then, along any geodesic $(\mu_t)_{t \in[0,1]}$ in $(P_2(\R^n), W_2)$, we have 
\[
\Ent_m(\mu_t) \leq (1-t)\Ent_m(\mu_0) + t\Ent_m(\mu_1) - \frac{\kappa}{2}t(1-t)W_2^2(\mu_0, \mu_1)
\]
for all $t\in[0,1]$. 
\end{description}
\end{Thm}

In order to describe the HWI inequality and logarithmic Sobolev inequality, we need the notion of the (relative) Fisher information. 
Let $m$ be a reference measure and $\mu \in P_2(\R^n)$. 
Then the {\it Fisher information of $\mu$ with respect to $m$} is defined by 
\begin{align*}
I_m(\mu) :=
\begin{cases}
 \int_{\R^n}\|\nabla \log \rho\|_2^2 \rho ~dm \quad &\text{if} \quad \mu=\rho m, \\
 \infty \quad &\text{if} \quad \mu\notll m.
\end{cases} 
\end{align*}
Clearly, the Fisher information is nonnegative. 

Theorem \ref{convexity} yields the HWI inequality and Talagrand inequality as follows.  
\begin{Thm}[HWI inequality]\label{HWI}
Let $m = e^{-V}\mathcal{L}^n$ be a probability measure on $\R^n$ such that $V \in C^{\infty}(\R^n)$ is $\kappa$-convex for some $\kappa>0$,  
and $\mu = \rho m\in P_2(\R^n)$ such that $\log\rho\in H^1(\R^n,\mu)$ and $\Ent_m(\mu)<\infty$ are satisfied. 
Then we have  
\[
\Ent_m(\mu) \leq W_2(\mu, m)\sqrt{I_m(\mu)} -\frac{\kappa}{2}W_2^2(\mu,m).
\]
\end{Thm}

\begin{proof}
By an approximation, we may assume that $\rho$ is smooth and compactly supported (see \cite[Theorem 9.17]{V1} for details).  
Substituting $\mu_0=\mu$ and $\mu_1=m$ in Theorem \ref{convexity}(2), we obtain that for any $t\in(0,1)$, 
\[
\Ent_m(\mu_t) \leq (1-t)\Ent_m(\mu) - \frac{\kappa}{2}t(1-t)W_2^2(\mu, m),
\]
where $(\mu_t)_{t\in[0,1]}$ is the geodesic from $\mu$ to $m$. Equivalently, we have 
\begin{align}\label{h}
\Ent_m(\mu)\leq \frac{\Ent_m(\mu)-\Ent_m(\mu_t)}{t} - \frac{\kappa}{2}(1-t)W_2^2(\mu, m).
\end{align}
Let $T$ be the optimal transport map from $\mu$ to $m$. Then we have $\mu_t=((1-t)id + tT)_{\#}\mu$, 
where $id$ is the identity map on $\R^n$. 
Denote the density of $\mu_t$ with respect to $m$ by $\rho_t$. Then we obtain 
\begin{align}
\Ent_m(\mu)-\Ent_m(\mu_t) &= \int_{\R^n}(\rho \log \rho - \rho_t\log\rho_t) ~dm \notag\\
&\leq \int_{\R^n}(\log\rho +1)(\rho - \rho_t) ~dm \label{hh}\\
&= \int_{\R^n}\log\rho ~d\mu - \int_{\R^n}\log\rho ~d\mu_t \notag\\
&= \int_{\R^n}(\log\rho(x) - \log\rho((1-t)x + tT(x))) ~d\mu(x) \notag\\
&= t\int_{\R^n}\inn{\nabla\log\rho(x)}{x-T(x)} ~d\mu(x) + o(t) \notag\\
&\leq t\left(\int_{\R^n}\|\nabla\log\rho(x)\|_2^2 ~d\mu(x)\right)^{\frac{1}{2}} \left(\int_{\R^n}\|x-T(x)\|_2^2 ~d\mu(x)\right)^{\frac{1}{2}}+ o(t) \label{hhh},
\end{align}
where we used $x\log x-y\log y\leq (\log x+1)(x-y)$ for any $x,y>0$ in (\ref{hh}) by the convexity of $x\log x$, and the Cauchy-Schwarz inequality in (\ref{hhh}). 
Hence, combining the inequality (\ref{hhh}) with (\ref{h}), and letting $t\to+0$, we have the desired inequality. 
\end{proof}

\begin{Thm}[Talagrand inequality]\label{TI}
Let $m = e^{-V}\mathcal{L}^n$ be a probability measure on $\R^n$ such that $V \in C^{\infty}(\R^n)$ is $\kappa$-convex for some $\kappa>0$. 
Then for any $\mu\in P_2(\R^n)$, we have 
\[
\frac{1}{2}W_2^2(\mu, m) \leq \frac{1}{\kappa} \Ent_m(\mu).
\]
\end{Thm}

\begin{proof}
Substituting $\mu_0=m$ and $\mu_1=\mu$ in Theorem \ref{convexity}(2), we obtain that for any $t\in(0,1)$, 
\[
\Ent_m(\mu_t) \leq t\Ent_m(\mu) - \frac{\kappa}{2}t(1-t)W_2^2(m,\mu),
\]
where $(\mu_t)_{t\in[0,1]}$ is the geodesic from $m$ to $\mu$. 
Since the relative entropy is nonnegative, the inequality above yields that 
\[
\frac{\kappa}{2}t(1-t)W_2^2(m,\mu) \leq t\Ent_m(\mu) .
\]
Dividing both sides of the inequality above by $t$, and letting $t\to+0$, we have the desired inequality. 
\end{proof}
In particular, it follows from the triangle inequality and the arithmetic-geometric mean inequality that 
\[
\frac{1}{2}W_2^2(\mu, \nu) \leq \frac{2}{\kappa} \left(\Ent_m(\mu) + \Ent_m(\nu) \right)
\]
for any $\mu,\nu\in P_2(\R^n)$, which is weaker than 
the symmetrized Talagrand inequality as in Theorem $\ref{STI}$. 
On the other hand, it is clear that substituting $m=\nu$ in Theorem \ref{STI} recovers Theorem \ref{TI} (up to the additional symmetry conditions). 

The HWI inequality (Theorem \ref{HWI}) yields the logarithmic Sobolev inequality as follows. 
\begin{Thm}[Logarithmic Sobolev inequality]\label{LSI}
Let $m = e^{-V}\mathcal{L}^n$ be a probability measure on $\R^n$ such that $V \in C^{\infty}(\R^n)$ is $\kappa$-convex for some $\kappa>0$, 
and $\mu = \rho m\in P_2(\R^n)$ such that $\log\rho\in H^1(\R^n,\mu)$ and $\Ent_m(\mu)<\infty$ are satisfied. 
Then we have 
\[
\Ent_m(\mu) \leq \frac{1}{2\kappa}I_m(\mu).
\]
\end{Thm}

\begin{proof}
Completing the square on the right hand side of the HWI inequality in Theorem \ref{HWI} with respect to $W_2(\mu,m)$, we have 
\[
\Ent_m(\mu) \leq -\frac{\kappa}{2}\left(W_2(\mu,m)-\frac{1}{\kappa}\sqrt{I_m(\mu)}\right)^2 + \frac{1}{2\kappa}I_m(\mu), 
\]
which immediately yields the desired inequality. 
\end{proof}
Note that the Talagrand inequality (Theorem \ref{TI}) also follows from the logarithmic Sobolev inequality (Theorem \ref{LSI}) (see \cite{V2}).
The logarithmic Sobolev inequality (Theorem \ref{LSI}) yields the following inequality. 
\begin{Thm}[Poincar\'{e} inequality]\label{PI}
Let $m = e^{-V}\mathcal{L}^n$ be a probability measure on $\R^n$ such that $V \in C^{\infty}(\R^n)$ is $\kappa$-convex for some $\kappa>0$. 
Then for any $f\in H^1(\R^n,m)$ satisfying $\int_{\R^n}f ~dm =0$, we have 
\[
\kappa\int_{\R^n}f^2 ~dm \leq \int_{\R^n} \|\nabla f\|_2^2 ~dm.
\]
\end{Thm}

\begin{proof}
By truncation, we assume $f\in L^{\infty}(\R^n)$. 
Set $\rho:=1+\varepsilon f$ for a small enough constant $\varepsilon>0$, and $\mu:=\rho m$.
Note that $\mu \in P_2(\R^n)$. By the logarithmic Sobolev inequality (Theorem \ref{LSI}), we have 
\begin{align}\label{i}
\int_{\R^n}\rho\log\rho ~dm \leq \frac{1}{2\kappa}\int_{\R^n}\|\nabla\log\rho\|_2^2\rho ~dm.
\end{align}
Expanding the left hand side of the inequality above at $\varepsilon=0$, we obtain 
\begin{align*}
\int_{\R^n}\rho\log\rho ~dm = \int_{\R^n}(1+\varepsilon f)\log(1+\varepsilon f) ~dm = \frac{\varepsilon^2}{2}\int_{\R^n}f^2 ~dm + o(\varepsilon^2). 
\end{align*}
On the other hand, expanding the right hand side yields that 
\begin{align*}
\frac{1}{2\kappa}\int_{\R^n}\|\nabla\log\rho\|_2^2\rho ~dm = \frac{1}{2\kappa}\int_{\R^n}\frac{\|\nabla\rho\|_2^2}{\rho} ~dm = \frac{\varepsilon^2}{2\kappa}\int_{\R^n}\frac{\|\nabla f\|_2^2}{1+\varepsilon f} ~dm.
\end{align*}
Hence dividing both sides of (\ref{i}) by $\varepsilon^2$, and letting $\varepsilon\to +0$, we have the desired inequality. 
\end{proof}
Note that the Poincar\'{e} inequality (Theorem \ref{PI}) also follows from the Talagrand inequality (Theorem \ref{TI}) (see \cite{V2}).

%%%%%%%%%%%%%%%%%%%%%%%%%%%%%%%%%%%%%%%%%%%%%%%%%%%%%%%%%%%%%%%%%%%%%%%%%%%%%%%%%%%%%%%%%%%%%%%%%%%%%%%%%%%%%%%%%%%%%%%%%%%%%%%%%%%%%%%%%%%%%%%%%%%%%%%%%%%%%%%%%%%%%%%%%%%%%%%%%%%%%%%%%%%%%%%%%%
%%%%%%%%%%%%%%%%%%%%%%%%%%%%%%%%%%%%%%%%%%%%%%%%%%%%%%%%%%%%%%%%%%%%%%%%%%%%%%%%%%%%%%%%%%%%%%%%%%%%%%%%%%%%%%%%%%%%%%%%%%%%%%%%%%%%%%%%%%%%%%%%%%%%%%%%%%%%%%%%%%%%%%%%%%%%%%%%%%%%%%%%%%%%%%%%%%
%%%%%%%%%%%%%%%%%%%%%%%%%%%%%%%%%%%%%%%%%%%%%%%%%%%%%%%%%%%%%%%%%%%%%%%%%%%%%%%%%%%%%%%%%%%%%%%%%%%%%%%%%%%%%%%%%%%%%%%%%%%%%%%%%%%%%%%%%%%%%%%%%%%%%%%%%%%%%%%%%%%%%%%%%%%%%%%%%%%%%%%%%%%%%%%%%%
\section{Variants of the symmetrized Talagrand inequality}

In this section, we consider the meaning of the condition $\bary(\nu)=0$ in Theorem \ref{STI}(2).
For a probability measure $\mu$ on $\R^n$, its {\it barycenter} is defined by $\int_{\R^n}x ~d\mu(x)$, and 
denoted by $\bary(\mu)$.
For a probability measure $\mu$ on $\R^n$ and $a\in\R^n$, we denote by $\mu_a$ the probability measure translated by $a$: for any Borel subset $A\subset\R^n$, $\mu_a(A) = \mu(A-a)$, where
$A-a := \{x-a ~|~ x\in A\}$. 
Note that the barycenter of $\mu_a$ is $\bary(\mu_a)=\bary(\mu) +a$ for any probability measure $\mu$ on $\R^n$ and any $a\in\R^n$.

\begin{Prop}\label{translation}
Let $\nu\in P_2(\R^n)$. For any probability measure $\mu\in P_2(\R^n)$ and any $a\in\R^n$, it holds that
\begin{align}\label{aaaaaa}
\Ent_{\gamma_n}(\mu_a) - \frac{1}{2}W_2^2(\mu_a, \nu) = \Ent_{\gamma_n}(\mu) - \frac{1}{2}W_2^2(\mu, \nu) + \inn{\bary(\nu)}{a}.
\end{align}
In particular, if a probability measure $\mu\in P_2(\R^n)$ satisfies $\Ent_{\gamma_n}(\mu)<\infty$, then 
$\Ent_{\gamma_n}(\mu_a) - W_2^2(\mu_a, \nu)/2$ is independent of $a$ in the hyperplane orthogonal 
to $\bary(\nu)$.
Here, when $\bary(\nu)=0$, the hyperplane orthogonal 
to $\bary(\nu)$ means whole $\R^n$.
Therefore, $\bary(\nu)=0$ if and only if 
$\Ent_{\gamma_n}(\mu_a) - W_2^2(\mu_a, \nu)/2$ is independent of $a\in \R^n$.
\end{Prop}

\begin{proof}
Firstly, we prove the following formula: for any $a\in\R^n$, it holds that
\begin{align}\label{c}
\frac{1}{2}W_2^2(\mu_a, \nu) = \frac{1}{2}W_2^2(\mu, \nu) + \inn{\bary(\mu)}{a} - \inn{\bary(\nu)}{a} + \frac{1}{2}\|a\|_2^2.
\end{align}

Fix $a\in \R^n$. Let $f, g\in C_b(\R^n)$ satisfy 
$f(x)+g(y)\leq \|x-y\|_2^2/2$ for any $x, y \in \R^n$ and define $F, G:\R^n\to \R$ by 
$F(x) := f(x+a) - \inn{x}{a}, G(y) := g(y) + \inn{y}{a} - \|a\|_2^2/2$, respectively. 
Then $F\in L^1(\mu), G\in L^1(\nu)$ and 
\begin{align*}
F(x) + G(y) &\leq \frac{1}{2}\|x+a-y\|_2^2-\inn{x}{a}+\inn{y}{a}-\frac{1}{2}\|a\|_2^2 \\
               &= \frac{1}{2}\|x-y\|_2^2
\end{align*}
holds for any $x, y\in\R^n$. Thus by the Kantorovich duality (Theorem \ref{KD}), it follows that
\begin{align*}
\frac{1}{2}W_2^2(\mu, \nu) &\geq \int_{\R^n}F(x) ~d\mu(x) + \int_{\R^n}G(y) ~d\nu(y) \\
                                  &=\int_{\R^n}f(x) ~d\mu_a(x) + \int_{\R^n}g(y) ~d\nu(y) - \inn{\bary(\mu)}{a} + \inn{\bary(\nu)}{a} - \frac{1}{2}\|a\|_2^2.
\end{align*}
Since functions $f, g$ were arbitrary in $C_b(\R^n)$, again using the Kantorovich duality (Theorem \ref{KD} and Remark \ref{rem}), we have
\begin{align}\label{b}
\frac{1}{2}W_2^2(\mu, \nu) \geq  \frac{1}{2}W_2^2(\mu_a, \nu) - \inn{\bary(\mu)}{a} + \inn{\bary(\nu)}{a} - \frac{1}{2}\|a\|_2^2.
\end{align}
Replacing $\mu$ by $\mu_a$, and $a$ by $-a$ in $(\ref{b})$, since $(\mu_a)_{-a} = \mu$, we have 
\begin{align*}
 \frac{1}{2}W_2^2(\mu_a, \nu) &\geq  \frac{1}{2}W_2^2(\mu, \nu) - \inn{\bary(\mu_a)}{-a} + \inn{\bary(\nu)}{-a} - \frac{1}{2}\|-a\|_2^2 \\
                                        &= \frac{1}{2}W_2^2(\mu, \nu) + \inn{\bary(\mu)}{a} - \inn{\bary(\nu)}{a} + \frac{1}{2}\|a\|_2^2.
\end{align*}
Hence, (\ref{c}) follows from this inequality and (\ref{b}).
 
We prove \eqref{aaaaaa}. By the definition of the relative entropy, \eqref{aaaaaa} is clear when $\mu$ is not absolutely continuous with respect to $\gamma_n$. Thus we may assume that $\mu\ll\gamma_n$. 
Let $\rho$ be 
the density of $\mu$ with respect to $\gamma_n$.
Then the density of $\mu_a$ with respect to $\gamma_n$ is $\rho(x-a)\exp(-\|x-a\|_2^2/2+\|x\|_2^2/2)$, which yields that 
\begin{align}\label{a}
\Ent_{\gamma_n}(\mu_a) &= \int_{\R^n}\left(\log \rho(x-a) -\frac{\|x-a\|_2^2}{2} +\frac{\|x\|_2^2}{2}\right) ~d\mu_a(x) \notag \\
 &= \int_{\R^n}\log\rho(x) ~d\mu(x) + \int_{\R^n}\left(\inn{x}{a} + \frac{\|a\|_2^2}{2}\right) ~d\mu(x) \notag \\
 &= \Ent_{\gamma_n}(\mu) + \inn{\bary(\mu)}{a} + \frac{1}{2}\|a\|_2^2.
\end{align}
Combining this equality with (\ref{c}), we have the desired equality. 
\end{proof}

Combining Proposition \ref{translation} with Theorem \ref{STI}(2), we obtain a generalization of the symmetrized Talagrand inequality as follows. 
\begin{Cor}\label{SSTI}
For any $\mu, \nu\in P_2(\R^n)$, we have 
\begin{align}\label{q}
\frac{1}{2}W_2^2(\mu, \nu)\leq \Ent_{\gamma_n}(\mu)+\Ent_{\gamma_n}(\nu) - \inn{\bary(\mu)}{\bary(\nu)}.
\end{align}
The equality holds if and only if two probability measures $\mu$ and $\nu$ are Gaussian such that their covariant matrices are 
inverse to each other. 
In particular, if $\inn{\bary(\mu)}{\bary(\nu)} \geq 0$, then 
\[
\frac{1}{2}W_2^2(\mu, \nu)\leq \Ent_{\gamma_n}(\mu)+\Ent_{\gamma_n}(\nu).
\]
\end{Cor} 

\begin{proof}
Set $a:=-\bary(\mu)$. Then by Proposition \ref{translation}, we have 
\begin{align*}
\Ent_{\gamma_n}(\mu_a) - \frac{1}{2}W_2^2(\mu_a, \nu) &= \Ent_{\gamma_n}(\mu) - \frac{1}{2}W_2^2(\mu, \nu) + \inn{\bary(\nu)}{a} \\
 &= \Ent_{\gamma_n}(\mu) - \frac{1}{2}W_2^2(\mu, \nu) - \inn{\bary(\nu)}{\bary(\mu)}.
\end{align*}
On the other hand, since $\bary(\mu_a)=0$, Theorem \ref{STI}(2) yields that  
\[
\Ent_{\gamma_n}(\mu_a) - \frac{1}{2}W_2^2(\mu_a, \nu) \geq -\Ent_{\gamma_n}(\nu),
\] 
which implies the former assertion. 

The latter assertion follows immediately from the former one. 
\end{proof}

%\begin{proof}%もっと簡単な証明
%$a:=-\bary(\nu)$ とおく．このとき$\bary(\nu_a)=0$ なので，定理$\ref{STI}$ から，
%\[
%\frac{1}{2}W_2^2(\mu_a, \nu_a)\leq \Ent_{\gamma_n}(\mu_a)+\Ent_{\gamma_n}(\nu_a)
%\] 
%となる．一方で$W_2(\mu_a, \nu_a) = W_2(\mu, \nu)$ と 
% $\Ent_{\gamma_n}(\mu_a) = \Ent_{\gamma_n}(\mu) + \inn{\bary(\mu)}{a} + \frac{1}{2}\|a\|^2$，
% $\Ent_{\gamma_n}(\nu_a) = \Ent_{\gamma_n}(\nu) + \inn{\bary(\nu)}{a} + \frac{1}{2}\|a\|^2$から
%命題の主張を得る．
%\end{proof}

Note that Corollary \ref{SSTI} generalizes Theorem \ref{STI}(2), and Corollary \ref{SSTI} is in fact a self-improvement of Theorem \ref{STI}(2). 
Moreover, \eqref{q} in Corollary \ref{SSTI} is invariant under translations of probability measures. 
Let $\mu, \nu\in P_2(\R^n)$ and $a, b\in\R^n$, and consider \eqref{q} for $\mu_a$ and $\nu_b$: 
\begin{align}\label{e}
\frac{1}{2}W_2^2(\mu_a, \nu_b)\leq \Ent_{\gamma_n}(\mu_a)+\Ent_{\gamma_n}(\nu_b) - \inn{\bary(\mu_a)}{\bary(\nu_b)}. 
\end{align}
Using (\ref{c}) twice for the left hand side, we have 
\begin{align*}
\frac{1}{2}W_2^2(\mu_a, \nu_b) &= \frac{1}{2}W_2^2(\mu, \nu_b) +\inn{\bary(\mu)}{a} - \inn{\bary(\nu_b)}{a} + \frac{1}{2}\|a\|_2^2 \notag\\
&= \frac{1}{2}W_2^2(\mu, \nu) + \inn{\bary(\nu)}{b} - \inn{\bary(\mu)}{b} +\frac{1}{2}\|b\|_2^2 \notag \\
&\hspace{1cm} +\inn{\bary(\mu)}{a} - \inn{\bary(\nu)}{a} - \inn{a}{b} + \frac{1}{2}\|a\|_2^2.
\end{align*}
On the other hand, using (\ref{a}) for the right hand side in \eqref{e}, we have 
\begin{align*}
&\Ent_{\gamma_n}(\mu_a)+\Ent_{\gamma_n}(\nu_b) - \inn{\bary(\mu_a)}{\bary(\nu_b)} \\
&= \Ent_{\gamma_n}(\mu) + \inn{\bary(\mu)}{a} +\frac{1}{2}\|a\|_2^2 + \Ent_{\gamma_n}(\nu) + \inn{\bary(\nu)}{b} + \frac{1}{2}\|b\|_2^2 \\
&\hspace{1cm}-\inn{\bary(\mu)}{\bary(\nu)} - \inn{\bary(\mu)}{b} -\inn{\bary(\nu)}{a} - \inn{a}{b}.
\end{align*}
Hence, the inequality \eqref{e} is equivalent to \eqref{q} in Corollary \ref{SSTI} for $\mu$ and $\nu$. 

Finally, we describe the relation to the main theorems in this paper.  
In Theorem \ref{STI}(2), the assumption of $\nu$ satisfying $\bary(\nu)=0$ is needed, and in Theorem \ref{STI}(1), 
the symmetry of $\nu$ is assumed which is stronger than $\bary(\nu)=0$. A probability measure on $\R^n$ is said to be {\it symmetric} if its density is even. 
Since Corollary \ref{SSTI} is proved by a self-improvement of Theorem \ref{STI}(2), 
Corollary \ref{SSTI} implies that the symmetrized Talagrand inequality under the condition $\inn{\bary(\mu)}{\bary(\nu)}\geq 0$ 
and the one under the condition $\bary(\nu)=0$ are essentially equivalent. 
In the present paper, our goals are to prove symmetrized Talagrand inequalities for general (not necessarily Gaussian  nor symmetric) probability measures under the condition $\bary(\mu)=\bary(\nu)$.

%%%%%%%%%%%%%%%%%%%%%%%%%%%%%%%%%%%%%%%%%%%%%%%%%%%%%%%%%%%%%%%%%%%%%%%%%%%%%%%%%%%%%%%%%%%%%%%%%%%%%%%%%%%%%%%%%%%%%%%%%%%%%%%%%%%%%%%%%%%%%%%%%%%%%%%%%%%%%%%%%%%%%%%%%%%%%%%%%%%%%%%%%%%%%%%%%
%%%%%%%%%%%%%%%%%%%%%%%%%%%%%%%%%%%%%%%%%%%%%%%%%%%%%%%%%%%%%%%%%%%%%%%%%%%%%%%%%%%%%%%%%%%%%%%%%%%%%%%%%%%%%%%%%%%%%%%%%%%%%%%%%%%%%%%%%%%%%%%%%%%%%%%%%%%%%%%%%%%%%%%%%%%%%%%%%%%%%%%%%%%%%%%%%
%%%%%%%%%%%%%%%%%%%%%%%%%%%%%%%%%%%%%%%%%%%%%%%%%%%%%%%%%%%%%%%%%%%%%%%%%%%%%%%%%%%%%%%%%%%%%%%%%%%%%%%%%%%%%%%%%%%%%%%%%%%%%%%%%%%%%%%%%%%%%%%%%%%%%%%%%%%%%%%%%%%%%%%%%%%%%%%%%%%%%%%%%%%%%%%%%
\section{Caffarelli's contraction theorem}
In this section, we introduce an important theorem to prove our main theorems.
This result was firstly proved in 2000 by Caffarelli as follows (\cite{C}): 
When $V$ is a smooth $1$-convex function on $\R^n$, $\mu=e^{-V}\mathcal{L}^n$ is a probability measure on $\R^n$, and $T$ is
the optimal transport map from the standard Gaussian measure $\gamma_n$ on $\R^n$ to $\mu$, then $T$ is a $1$-Lipschitz map. 
This theorem is called Caffarelli's contraction theorem on which there are a lot of alternative proofs and advanced studies (see \cite{V07}, \cite{KM}, \cite{M}, \cite{FGP}).
In the present paper, we prove an extended version, using the original Caffarelli's paper \cite{C} for reference. 
This result is well-known to experts, but the author could not find the complete proof in literatures, hence we give a proof here. 
Its claim is as follows. 

\begin{Thm}[Caffarelli's contraction theorem]\label{CCT}
Let $\mu_i = e^{-V_i}\mathcal{L}^n \in P_2(\R^n)$ $(i = 0,1)$ such that $V_i$ $(i=0, 1)$ is a smooth function on $\R^n$, 
and $T:\R^n\to\R^n$ be the optimal transport map from $\mu_0$ to $\mu_1$.
If for some $\kappa_0, \kappa_1 > 0$, $V_i$ $(i=0, 1)$ satisfy $\nabla^2V_0 \leq \kappa_0$ and $\nabla^2 V_1 \geq  \kappa_1$, then it holds that 
\[
\sup_{x\in\R^n} \nabla T(x) \leq \sqrt{\frac{\kappa_0}{\kappa_1}}.
\]
In particular, $T$ is a $\sqrt{\kappa_0/\kappa_1}$-Lipschitz map.
\end{Thm}

In order to prove Theorem \ref{CCT}, we need the following lemma which is a modification from the original lemma (see \cite{C}, \cite{V07}). 
\begin{Lem}\label{lim}
Let $\mu_i$ $(i = 0,1)$ be probability measures which are absolutely continuous with respect to the Lebesgue measure on $\R^n$.
Assume that $\mu_0$ and $\mu_1$ satisfy $\mathrm{supp}(\mu_0)=\R^n$, $\mathrm{supp}(\mu_1) =\mathrm{B}_r$ for some $r>0$, 
and the density $g$ of $\mu_1$ is continuous on $\mathrm{B}_r$ and satisfies $\min\{g(x)~|~x\in\mathrm{B}_r\}>0$,
where $\mathrm{B}_r$ is the closed ball in $\R^n$ whose center is the origin and radius is $r$. 
Let $T$ be the optimal transport map from $\mu_0$ to $\mu_1$. Then it holds that 
\[
\lim_{\|x\|_2\to\infty}\left(T(x) - r\frac{x}{\|x\|_2}\right) = 0.
\]
\end{Lem}

\begin{proof}
Fix $x\in\R^n\setminus \{0\}$ and $\theta\in(0,\pi/2)$, and set $y:=T(x)$ and $\Gamma_{y,\theta}:=\{y'\in\R^n\setminus \{ y \} ~|~\angle(x,y'-y)\leq \theta \}$, 
where $\angle(w,z)\in[0,\pi]$ is the angle between $w$ and $z$ in $\R^n\setminus\{ 0 \}$. 
Now, fix $y'\in\Gamma_{y,\theta}\cap\mathrm{B}_r$, and take $x'\in\R^n\setminus\{ x \}$ satisfying $y'=T(x')$. 
Since ${\rm supp}((id, T)_{\#}\mu_0)\subset \R^n\times\R^n$ is cyclically monotone (Theorem \ref{CM}), it follows that $\inn{x'-x}{y'-y}\geq 0$, which implies that 
$\angle(x'-x,y'-y)\leq \pi/2$. 
Therefore, it follows from the triangle inequality for the angle that $\angle(x,x'-x)\leq\pi/2+\theta$.
Hence, setting
\[
\Gamma_{x,\theta}:=\left\{z\in\R^n\setminus\{ x \} ~{\Big |}~\angle(x,z-x) \leq \frac{\pi}{2}+\theta\right\},
\]
we have $x'\in\Gamma_{x,\theta}$. 
Since $y'\in\Gamma_{y,\theta}\cap\mathrm{B}_r$ was arbitrary, it follows that $T^{-1}(\Gamma_{y,\theta} \cap\mathrm{B}_r)\subset \Gamma_{x,\theta}$. Thus we have 
\begin{align}\label{f}
\min_{z\in\mathrm{B}_r} g(z) \cdot\mathcal{L}^n(\Gamma_{y,\theta}\cap\mathrm{B}_r) \leq \mu_1(\Gamma_{y,\theta}\cap\mathrm{B}_r) = \mu_0(T^{-1}(\Gamma_{y,\theta}\cap\mathrm{B}_r)) \leq \mu_0(\Gamma_{x,\theta}).
\end{align}
On the other hand, $\min_{z\in\mathrm{B}_r} g(z)>0$ by the assumption of $g$, and $\lim_{\|x\|_2\to\infty}\mu_0(\Gamma_{x,\theta})=0$ by the definition. 
Therefore, it follows from ($\ref{f}$) that $\lim_{\|x\|_2\to\infty} \mathcal{L}^n(\Gamma_{y,\theta}\cap\mathrm{B}_r)=0$. 
Finally, letting $\theta\to\pi/2$, we obtain the claim. 
\end{proof}

\begin{proof}[Proof of Theorem \ref{CCT}]
By taking an approximating sequence (for instance replacing $\mu_1$ by the normalization of $\mu_1|_{B_r}$ for $r>0$), 
we may assume that $\mu_0, \mu_1$ satisfy the assumptions in Lemma \ref{lim} (see \cite[Exercise 2.17]{V1}). 
Let $\phi$ be a Kantorovich potential such that $T=\nabla\phi$. 
Fix $h>0$, and for $x\in\R^n$ and $e\in\Sbb^{n-1}$, set $\delta\phi(x,e)=\delta\phi_e(x):=\phi(x+he)+\phi(x-he)-2\phi(x)$.

{\bf Step 1.} On $\R^n$, we prove  
\begin{align*}
\nabla T=\nabla^2\phi \leq 2\sqrt{\frac{\kappa_0}{\kappa_1}}.
\end{align*}

Since $\phi$ is convex, it follows that 
\[
0\leq\delta\phi_e(x)\leq \inn{\nabla\phi(x+he)}{he} + \inn{\nabla\phi(x-he)}{-he}=h\inn{T(x+he)}{e} - h\inn{T(x-he)}{e}
\]
for any $x\in\R^n$ and any $e\in\Sbb^{n-1}$. Since $\lim_{\|x\|_2\to\infty}(\inn{T(x+he)}{e} - \inn{T(x-he)}{e})=0$ by Lemma $\ref{lim}$,
$\delta\phi_e$ has a maximizing point for any $e\in\Sbb^{n-1}$.
Now, let $(x_0,e_0)\in\R^n\times\Sbb^{n-1}$ be a point attaining the maximum of $\delta\phi$. Then we have 
\begin{align}
\nabla\phi(x_0+he_0) + \nabla\phi(x_0-he_0) - 2\nabla\phi(x_0)=0 \label{g}, 
\end{align}
and
\begin{align}
\inn{\nabla\phi(x_0+he_0)}{v}-\inn{\nabla\phi(x_0-he_0)}{v} = 0 \quad \forall v\in {e_0}^{\perp} \label{gg}, 
\end{align}
where ${e_0}^{\perp}$ is the subspace in $\R^n$ whose elements are perpendicular to $e_0$. 
(\ref{gg}) yields that there exists $\alpha\in\R$ such that $\nabla\phi(x_0+he_0) - \nabla\phi(x_0-he_0) = \alpha e_0$. 
Combining this with (\ref{g}), we have 
\begin{align}\label{gggg}
\nabla\phi(x_0\pm he_0) =\nabla\phi(x_0) \pm \frac{\alpha}{2} e_0.
\end{align}
On the other hand, by the Monge-Amp\`{e}re equation (Theorem \ref{MA}), it follows that 
\begin{align}\label{ggg}
&\log\det \nabla^2\phi(x_0+he_0) +\log\det \nabla^2\phi(x_0-he_0) -2\log\det \nabla^2\phi(x_0) \notag \\
&=V_1(\nabla\phi(x_0+he_0)) + V_1(\nabla\phi(x_0-he_0)) -2V_1(\nabla\phi(x_0)) - V_0(x_0+he_0) -V_0(x_0-he_0) + 2V_0(x_0).
\end{align}
Since $\log\det$ is concave (see Lemma \ref{logdet} below),
\[
\log\det \nabla^2\phi(x_0+he_0) +\log\det \nabla^2\phi(x_0-he_0) -2\log\det \nabla^2\phi(x_0) \leq \inn{\nabla[\log\det] \left(\nabla^2\phi(x_0)\right)}{\nabla^2\delta\phi_{e_0}(x_0)}.
\]
$\nabla[\log\det] \left(\nabla^2\phi(x_0)\right)$ equals $\nabla^2\phi(x_0)^{-1}$ which is a positive definite symmetric matrix, 
and $\nabla^2\delta\phi_{e_0}(x_0)$ is a negative semi-definite symmetric matrix by the definition of $x_0$. Hence, 
the left hand side of (\ref{ggg}) is nonpositive. 
Moreover, since $\nabla^2V_0\leq\kappa_0$ and $\nabla^2V_1\geq \kappa_1$ by assumptions, 
combining this with (\ref{gggg}), we obtain 
\[
V_0(x_0+he_0) +V_0(x_0-he_0) - 2V_0(x_0) \leq \kappa_0h^2,
\]
and
\begin{align*}
&V_1(\nabla\phi(x_0+he_0)) + V_1(\nabla\phi(x_0-he_0)) -2V_1(\nabla\phi(x_0)) \\
&= V_1\left(\nabla\phi(x_0) + \frac{\alpha}{2}e_0\right) + V_1\left(\nabla\phi(x_0) - \frac{\alpha}{2} e_0\right) -2V_1(\nabla\phi(x_0)) \\
&\geq  \frac{\kappa_1\alpha^2}{4}. 
\end{align*}
Thus by (\ref{ggg}), it follows that $|\alpha|\leq 2h\sqrt{\kappa_0/\kappa_1}$. 
Again, by the convexity of $\phi$ and (\ref{gggg}), for any $(x,e)\in\R^n\times\Sbb^{n-1}$, we have 
\begin{align}\label{ggggg}
\delta\phi_e(x)\leq \delta\phi_{e_0}(x_0) &\leq \inn{\nabla\phi(x_0+he_0)}{he_0} + \inn{\nabla\phi(x_0-he_0)}{-he_0} \notag\\
& = \inn{\nabla\phi(x_0)+\frac{\alpha}{2} e_0}{he_0} + \inn{\nabla\phi(x_0)-\frac{\alpha}{2} e_0}{-he_0} \notag \\
& = \alpha h \notag \\
&\leq 2\sqrt{\frac{\kappa_0}{\kappa_1}}h^2.
\end{align}
Finally, dividing both sides above by $h^2$ and letting $h\to +0$, we have the desired inequality. 

{\bf Step 2.} We prove the main claim. Set $M:=\sqrt{\kappa_0/\kappa_1}$, 
$a_1:=2$, and assume that for some $a_k>1$ $(k\in\mathbb{N})$, $\nabla T=\nabla^2\phi \leq a_kM$ is satisfied on $\R^n$. 
Now, by the definition of $\delta\phi$, it follows that 
\begin{align}\label{gggggg}
\delta\phi_{e_0}(x_0)=\int_0^h\inn{\nabla\phi(x_0+te_0)-\nabla\phi(x_0-te_0)}{e_0} ~dt.
\end{align}
Then by the assumption above, for any $t\in[0,h]$, we obtain 
\begin{align*}
\inn{\nabla\phi(x_0+te_0)-\nabla\phi(x_0-te_0)}{e_0}&= \int_0^t\inn{\nabla^2\phi(x_0+se_0)\cdot e_0+\nabla^2\phi(x_0-se_0)\cdot e_0}{e_0} ~ds \\
&\leq 2a_kMt.
\end{align*}
On the other hand, since $\phi$ is convex, $t\mapsto \inn{\nabla\phi(x_0+te_0)-\nabla\phi(x_0-te_0)}{e_0}$ is monotone increasing. 
Hence for any $t\in[0,h]$, by the same computations as \eqref{ggggg}, we have 
\begin{align*}
\inn{\nabla\phi(x_0+te_0)-\nabla\phi(x_0-te_0)}{e_0}&\leq \inn{\nabla\phi(x_0+he_0)-\nabla\phi(x_0-he_0)}{e_0}\\
&\leq 2Mh.
\end{align*}
Therefore, for any $(x,e)\in\R^n\times\Sbb^{n-1}$, it follows from (\ref{gggggg}) that 
\begin{align*}
\delta\phi_e(x)&\leq \delta\phi_{e_0}(x_0)\\
&= \int_0^h\inn{\nabla\phi(x_0+te_0)-\nabla\phi(x_0-te_0)}{e_0} ~dt \\
&\leq \int_0^h \min\{2a_kMt,2Mh\} ~dt\\
&= \frac{2a_k-1}{a_k}Mh^2,
\end{align*}
and dividing both sides above by $h^2$, and letting $h\to +0$, we have 
\[
\nabla^2\phi \leq \frac{2a_k-1}{a_k}M
\]
on $\R^n$. Set $a_{k+1}:=(2a_k-1)/a_k$, then the positive sequence $(a_k)_{k\in \mathbb{N}}$ $(a_1=2)$ satisfies
$\lim_{k\to\infty}a_k=1$, and hence the desired inequality follows. 

The assertion that $T$ is $\sqrt{\kappa_0/\kappa_1}$-Lipschitz follows from for any $x,y\in\R^n$, 
\begin{align*}
\| T(y)-T(x)\|_2^2 &= \int_0^1\int_0^1 \inn{\nabla T((1-t)x+ty)\cdot (y-x)}{\nabla T((1-s)x+sy)\cdot (y-x)} ~dsdt \\
&\leq \frac{\kappa_0}{\kappa_1} \| y-x\|_2^2
\end{align*}
by the former claim. 
\end{proof}

%%%%%%%%%%%%%%%%%%%%%%%%%%%%%%%%%%%%%%%%%%%%%%%%%%%%%%%%%%%%%%%%%%%%%%%%%%%%%%%%%%%%%%%%%%%%%%%%%%%%%%%%%%%%%%%%%%%%%%%%%%%%%%%%%%%%%%%%%%%%%%%%%%%%%%%%%%%%%%%%%%%%%%%%%%%%%%%%%%%%%%%%%%%%%%%%%%
%%%%%%%%%%%%%%%%%%%%%%%%%%%%%%%%%%%%%%%%%%%%%%%%%%%%%%%%%%%%%%%%%%%%%%%%%%%%%%%%%%%%%%%%%%%%%%%%%%%%%%%%%%%%%%%%%%%%%%%%%%%%%%%%%%%%%%%%%%%%%%%%%%%%%%%%%%%%%%%%%%%%%%%%%%%%%%%%%%%%%%%%%%%%%%%%%%
%%%%%%%%%%%%%%%%%%%%%%%%%%%%%%%%%%%%%%%%%%%%%%%%%%%%%%%%%%%%%%%%%%%%%%%%%%%%%%%%%%%%%%%%%%%%%%%%%%%%%%%%%%%%%%%%%%%%%%%%%%%%%%%%%%%%%%%%%%%%%%%%%%%%%%%%%%%%%%%%%%%%%%%%%%%%%%%%%%%%%%%%%%%%%%%%%%
\section{Main theorems and their applications to functional inequalities}
In this section, we prove Theorems \ref{MT1} and \ref{MT2}. The basic idea comes from \cite{CFP} by Courtade, Fathi and Pananjady, in which  
they proved the refined entropy power inequality.
In addition, we give a new and direct proof of Theorem $\ref{STI}$ on the real line in Subsection 5.3. 

%%%%%%%%%%%%%%%%%%%%%%%%%%%%%%%%%%%%%%%%%%%%%%%%%%%%%%%%%%%%%%%%%%%%%%%%%%%%%%%%%%%%%%%%%%%%%%%%%%%%%%%%%%%%%%%%%%%%%%%%%%%%%%%%%%%%%%%%%%%%%%%%%%%%%%%%%%%%%%%%%%%%%%%%%%%%%%%%%%%%%%%%%%%%%%%%%%
\subsection{Proof of Theorem \ref{MT1}}

In order to prove the main theorems, we need the next lemma, which is also used in the next subsection. 
For simplicity, we introduce the following notations. 
For any positive definite symmetric matrix $P\in \R^{n\times n}$, we denote its maximal eigenvalue by $\lambda_P$, 
and the identity matrix by $I_n \in \R^{n\times n}$. 
\begin{Lem}[\cite{CFP}]\label{logdet}
Let $A, B \in \R^{n\times n}$ be positive definite symmetric matrices. 
Then for any $t\in[0,1]$, we have 
\[
\log\det\left((1-t)A + tB\right) \geq (1-t)\log\det A + t\log\det B + \frac{1}{2}t(1-t)\frac{1}{\max\{\lambda_A^2, \lambda_B^2\}}\|A - B\|_{HS}^2,
\]
where $\|\cdot\|_{HS}$ is the Hilbert-Schmidt norm. 
\end{Lem}

%\begin{proof}
%Let $\mathrm{Sym}^{+}(n)$ be the set of all $n$-order square positive definite symmetric matrices, 
%and for $X\in \mathrm{Sym}^{+}(n)$, set $F(X):=\log\det X$. 
%Then it follows that $\nabla F(X)=X^{-1}$ and $\nabla^2 F(X)=-X^{-1}\otimes X^{-1}$, 
%where $\otimes$ is the Kronecker product. 
%For any $X\in \mathrm{Sym}^{+}(n)$, $X^{-1}\otimes X^{-1}$ is $n^2$-order square positive definite symmetric matrix, and 
%$\lambda_{-X^{-1}\otimes {X^{-1}}}= -1/\lambda_X^2$ holds. 
%Thus for any $A,B\in \mathrm{Sym}^{+}(n)$ and any $t\in[0,1]$, 
%it follows that $\nabla^2 F((1-t)A + tB) \leq - 1/\lambda_{(1-t)A + tB}^2$. 
%Moreover, since it follows from the theory of linear algebra that $\lambda_{(1-t)A + tB}$ is convex in $t$, we have 
%$\nabla^2 F((1-t)A + tB) \leq - 1/\max \{\lambda_A^2,\lambda_B^2\}$ for any $t\in[0,1]$. 
%Now, for $A,B\in \mathrm{Sym}^{+}(n)$, set $K(A,B):=1/\max \{\lambda_A^2,\lambda_B^2\}$. 
%Since $\nabla^2 F((1-t)A + tB) \leq - K(A,B)$, it follows that for any $t\in[0,1]$, 
%\[
%F(A) \leq F((1-t)A + tB) + \inn{\nabla F((1-t)A + tB)}{A-B}t - \frac{t^2K(A,B)}{2}\|A-B\|_{HS}^2,
%\]
%and
%\[
%F(B) \leq F((1-t)A + tB) - \inn{\nabla F((1-t)A + tB)}{A-B}(1-t) - \frac{(1-t)^2K(A,B)}{2}\|A-B\|_{HS}^2,
%\]
%which imply that 
%\[
%(1-t)F(A) + tF(B) \leq  F((1-t)A + tB) - \frac{t(1-t)K(A,B)}{2}\|A-B\|_{HS}^2,
%\]
%and so the claim is proven.
%\end{proof}

\begin{proof}[Proof of Theorem \ref{MT1}]
Let $T$ be the optimal transport map from $\mu_0$ to $\mu_1$. Then the geodesic $(\mu_t)_{t\in [0,1]}$ from $\mu_0$ to $\mu_1$ is represented as $\mu_t = ((1-t)id + tT)_{\#}\mu_0$,
where $id$ is the identity map on $\R^n$.
Let $\rho_t$ be the density of $\mu_t$ with respect to the Lebesgue measure. Then 
by the Monge-Amp\`{e}re equation (Theorem \ref{MA}), we have, for any $t\in[0,1]$ and $\Lc^n$-a.e. $x\in \R^n$, 
\[
\rho_0(x) = \rho_t((1-t)x + tT(x))\det((1-t)I_n + t\nabla T(x)).
\]
Hence, for any $t\in [0,1]$, we obtain 
\begin{align}\label{aa}
\Ent (\mu_t) &= \int_{\R^n}\log\rho_t ~d\mu_t \notag \\
 &= \int_{\R^n} \log\rho_t((1-t)x + tT(x)) ~d\mu_0(x) \notag \\
 &= \int_{\R^n} (\log\rho_0(x) - \log\det((1-t)I_n + t\nabla T(x))) ~d\mu_0(x)\notag  \\
 &= \Ent (\mu_0) - \int_{\R^n}\log\det((1-t)I_n + t\nabla T(x)) ~d\mu_0(x).
\end{align}
By Lemma \ref{logdet}, for $\Lc^n$-a.e. $x\in \R^n$, we have 
\begin{align}\label{aaa}
\log\det((1-t)I_n + t\nabla T(x)) \geq t\log\det\nabla T(x) + \frac{1}{2}t(1-t)\frac{1}{\max\left\{1, \lambda_{\nabla T(x)}^2\right\}}\|I_n-\nabla T(x)\|_{HS}^2.
\end{align}
On the other hand, by Caffarelli's contraction theorem (Theorem \ref{CCT}) and the assumption (ii), we have 
\begin{align}\label{aaaa}
\sup_{x\in\R^n} \lambda_{\nabla T(x)} \leq \sqrt{\frac{\kappa_0}{\kappa_1}}.
\end{align}
Therefore, combining (\ref{aa}) with (\ref{aaa}) and (\ref{aaaa}), we obtain 
\begin{align}\label{aaaaa}
\Ent (\mu_t) &\leq \Ent (\mu_0) - t\int_{\R^n}\log\det\nabla T ~d\mu_0 \notag \\
&\hspace{4cm}
- \frac{1}{2\max\left\{1, \frac{\kappa_0}{\kappa_1}\right\}}t(1-t)\int_{\R^n}\|I_n-\nabla T(x)\|_{HS}^2 ~d\mu_0(x) \notag \\
                &= (1-t)\Ent (\mu_0) + t\Ent (\mu_1) - \frac{1}{2}\min\left\{1, \frac{\kappa_1}{\kappa_0}\right\}t(1-t)\int_{\R^n}\|I_n-\nabla T(x)\|_{HS}^2~d\mu_0(x),
\end{align}
where we used the Monge-Amp\`{e}re equation (Theorem \ref{MA}) in the last equality. Moreover, it follows from the assumption (i) that 
$\int_{\R^n}(x - T(x))~d\mu_0(x) = \bary(\mu_0) - \bary(\mu_1) = 0$, which implies that by the assumption (iii), 
\[
C_{\mu_0}\int_{\R^n}(x_i - T_i(x))^2 ~d\mu_0(x) \leq \int_{\R^n}\|\nabla\left(x_i - T_i(x)\right)\|_2^2~d\mu_0(x)
\]
for $i = 1,2, \dots, n$, where $x_i$ and $T_i$ are the $i$-th components of $x$ and $T$, respectively. 
Summing up in $i$, 
we have 
\begin{align*}
C_{\mu_0}\int_{\R^n}\|x - T(x)\|_2^2~d\mu_0(x) &\leq \int_{\R^n}\|\nabla\left(x - T(x)\right)\|_{HS}^2~d\mu_0(x) \\
                                              &= \int_{\R^n}\|I_n - \nabla T(x)\|_{HS}^2~d\mu_0(x).
\end{align*}
Combining this inequality with (\ref{aaaaa}), by optimality of $T$, we have 
\begin{align*}
\Ent (\mu_t) &\leq (1-t)\Ent (\mu_0) + t\Ent (\mu_1) - \frac{1}{2}C_{\mu_0}\min\left\{1, \frac{\kappa_1}{\kappa_0}\right\}t(1-t)\int_{\R^n}\|x - T(x)\|_2^2~d\mu_0(x) \\
                 &= (1-t)\Ent (\mu_0) + t\Ent (\mu_1) - \frac{1}{2}C_{\mu_0}\min\left\{1, \frac{\kappa_1}{\kappa_0}\right\}t(1-t)W_2^2(\mu_0, \mu_1),
\end{align*}
which implies that $\Ent$ is $C_{\mu_0}\min\left\{1, \kappa_1/\kappa_0\right\}$-convex along $(\mu_t)_{t\in [0, 1]}$ in $(P_2(\R^n),W_2)$.
\end{proof}

\begin{Cor}\label{C1}
Let $\mu_i$ $(i = 0,1)$ and $(\mu_t)_{t\in [0, 1]}$ be as in Theorem \ref{MT1}, 
and $m = e^{-V}\mathcal{L}^n$ be a probability measure on $\R^n$ such that $V \in C^{\infty}(\R^n)$ is $\kappa$-convex for some $\kappa>0$. 
Then $\Ent_{m}$ is $\left(\kappa + C_{\mu_0}\min\left\{1, \kappa_1/\kappa_0 \right\}\right)$-convex along $(\mu_t)_{t\in [0, 1]}$ in $(P_2(\R^n),W_2)$.
\end{Cor}

\begin{proof}
In general, for a probability measure $\zeta=\rho m$ on $\R^n$, it holds that 
\[
\Ent_m(\zeta)=\int_{\R^n}\log\rho ~d\zeta = \int_{\R^n}\log(\rho e^{-V}) ~d\zeta + \int_{\R^n} V ~d\zeta
=\Ent(\zeta)  + \int_{\R^n} V ~d\zeta.
\]
Hence, we obtain $\Ent_m(\mu_t) = \Ent(\mu_t) + \int_{\R^n}V ~d\mu_t$ for all $t\in [0,1]$. 
Since it follows from Theorem $\ref{MT1}$ that the first term of the right hand side is $C_{\mu_0}\min\left\{1, \kappa_1/\kappa_0\right\}$-convex along $(\mu_t)_{t\in [0, 1]}$, 
it suffices to prove that the second term is $\kappa$-convex along $(\mu_t)_{t\in [0, 1]}$, 
which is easily proved by the $\kappa$-convexity of $V$.
\end{proof}

\begin{proof}[Proof of Corollary \ref{MC1}]
It follows from Corollary \ref{C1} that for any $t\in[0,1]$, 
\[
\Ent_m (\mu_t) \leq (1-t)\Ent_m (\mu_0) + t\Ent_m (\mu_1) - \frac{1}{2}\left(\kappa + C_{\mu_0}\min\left\{1, \frac{\kappa_1}{\kappa_0}\right\}\right)t(1-t)W_2^2(\mu_0, \mu_1),
\]
where $(\mu_t)_{t\in [0,1]}$ is the geodesic from $\mu_0$ to $\mu_1$ in $(P_2(\R^n),W_2)$. 
The first inequality is proved by $\Ent_m(\mu_t) \geq 0$ in the case of $t=1/2$.

Then the remaining assertions immediately follow combined with Theorem \ref{PI}.
\end{proof}

Moreover, by similar proofs to Theorems \ref{HWI} and \ref{LSI}, Corollary \ref{C1} yields refined functional inequalities as follows.
\begin{Cor}\label{C13}
Let $m = e^{-V}\mathcal{L}^n$ be a probability measure on $\R^n$ such that $V \in C^{\infty}(\R^n)$ is $\kappa$-convex for some $\kappa>0$, and $\mu=e^{-W}\Lc^n\in P_2(\R^n)$ with $\Ent_m(\mu)<\infty$
where $W$ is a smooth function on $\R^n$. Moreover we assume the following three conditions: 
\begin{description}
\item{(i)} $\bary(m)=\bary(\mu)$.
\item{(ii)} $\nabla^2 W \leq \tau$ holds for some $\tau>0$.
\item{(iii)} $\mu$ satisfies the Poincar\'{e} inequality with $C_{\mu}>0$.
\end{description}
Then we have 
\begin{description}
\item{(1)} (HWI inequality) 
\[
\Ent_m(\mu) \leq W_2(\mu, m)\sqrt{I_m(\mu)} - \left(\frac{\kappa}{2} + \frac{C_{\mu}}{2}\min\left\{1, \frac{\kappa}{\tau} \right\}\right)W_2^2(\mu,m).
\]
\item{(2)} (logarithmic Sobolev inequality)
\[
\Ent_m(\mu) \leq \frac{1}{2}\cdot\frac{1}{\kappa + C_{\mu}\min\left\{1, \frac{\kappa}{\tau} \right\}} I_m(\mu).
\]
\end{description}
\end{Cor}

%%%%%%%%%%%%%%%%%%%%%%%%%%%%%%%%%%%%%%%%%%%%%%%%%%%%%%%%%%%%%%%%%%%%%%%%%%%%%%%%%%%%%%%%%%%%%%%%%%%%%%%%%%%%%%%%%%%%%%%%%%%%%%%%%%%%%%%%%%%%%%%%%%%%%%%%%%%%%%%%%%%%%%%%%%%%%%%%%%%%%%%%%%%%%%%%%
\subsection{Proof of Theorem \ref{MT2}} \label{GG}

Before giving the proof, we define the notion of generalized geodesics, which was introduced in \cite{AGS}. 
\begin{Def}
Let $\mu_i \in P_2(\R^n)$ $(i = 0,1,2)$ be probability measures on $\R^n$ such that $\mu_2$ is absolutely continuous with respect to the Lebesgue measure, 
and let $T_0,T_1$ be the optimal transport maps from $\mu_2$ to $\mu_0,\mu_1$, respectively. 
Then for any $t\in [0,1]$, we define the probability measure $\mu_t$ by 
$((1-t)T_0 + tT_1)_{\#}\mu_2$, and call 
$(\mu_t)_{t\in [0,1]}$ the {\it generalized geodesic from $\mu_0$ to $\mu_1$ with the base $\mu_2$}.
\end{Def}

In particular, when $\mu_0=\mu_2$ (or $\mu_1=\mu_2$), the generalized geodesic with the base $\mu_2$ coincides with the geodesic in $(P_2(\R^n),W_2)$ in the usual sense. 

\begin{proof}[Proof of Theorem \ref{MT2}]
Let $T_0,T_1$ be the optimal transport maps from $m$ to $\mu_0,\mu_1$, respectively. Then for any $t\in [0,1]$, 
it follows from the definition of $\mu_t$ that $\mu_t = ((1-t)T_0 + tT_1)_{\#}m$. 
Now, we denote the density of $\mu_t$ with respect to $m$ by $\rho_t$. Then 
by the Monge-Amp\`{e}re equation (Theorem \ref{MA}), for any $t\in[0,1]$ and $m$-a.e. $x\in \R^n$, we have  
\[
e^{-V(x)} = \rho_t((1-t)T_0(x) + tT_1(x))e^{-V((1-t)T_0(x) + tT_1(x))} \det((1-t)\nabla T_0(x) + t\nabla T_1(x)),
\]
which implies that for any $t\in [0,1]$,  
\begin{align*}
\Ent_m (\mu_t) &= \int_{\R^n}\log\rho_t ~d\mu_t \notag \\
 &= \int_{\R^n}\log\rho_t((1-t)T_0 + tT_1) ~dm \notag \\
 &= \int_{\R^n}( - V(x) + V((1-t)T_0(x) + tT_1(x))) ~dm(x) \notag \\
 &\hspace{2cm} - \int_{\R^n}\log\det((1-t)\nabla T_0(x) + t\nabla T_1(x)) ~dm(x) \notag \\ 
 &= - \int_{\R^n}V ~dm + \int_{\R^n}V((1-t)T_0(x) + tT_1(x)) ~dm(x)  \notag \\
 &\hspace{2cm}- \int_{\R^n}\log\det((1-t)\nabla T_0(x) + t\nabla T_1(x))~dm(x).
\end{align*}
Thus by the $\kappa$-convexity of $V$, it follows that 
\begin{align}\label{bb}
(&1-t)\Ent_m(\mu_0) + t\Ent_m(\mu_1) - \Ent_m(\mu_t) \notag \\
=& (1-t)\int_{\R^n}V(T_0(x))~dm(x) + t\int_{\R^n}V(T_1(x))~dm(x) - \int_{\R^n}V((1-t)T_0(x) + tT_1(x))~dm(x) \notag \\
 &\hspace{1cm} + \int_{\R^n}\log\det((1-t)\nabla T_0(x) + t\nabla T_1(x))~dm(x)
 - (1-t)\int_{\R^n}\log\det\nabla T_0(x) ~dm(x) \notag \\
 &\hspace{2cm}- t\int_{\R^n}\log\det\nabla T_1(x) ~dm(x) \notag \\
 \geq & \frac{\kappa}{2}t(1-t)\int_{\R^n}\|T_0 - T_1\|_2^2~dm \notag \\
 &+ \int_{\R^n}\left(\log\det((1-t)\nabla T_0(x) + t\nabla T_1(x)) - (1-t)\log\det\nabla T_0(x) - t\log\det\nabla T_1(x)\right)~dm(x).
\end{align}

On the other hand, since $(T_0, T_1)_{\#}m \in \Pi(\mu_0, \mu_1)$, it follows that 
$\int_{\R^n}\|T_0 - T_1\|_2^2 ~dm=\int_{\R^n}\|x - y\|_2^2 ~d(T_0, T_1)_{\#}m(x, y) \geq W_2^2(\mu_0, \mu_1)$. 
Moreover by Lemma \ref{logdet}, the assumption (ii) and Caffarelli's contraction theorem (Theorem \ref{CCT}), 
we obtain  
\begin{align*}
&\log\det((1-t)\nabla T_0(x) + t\nabla T_1(x)) - (1-t)\log\det\nabla T_0(x) - t\log\det\nabla T_1(x) \\
&\geq \frac{1}{2}t(1-t)\min\left\{\frac{\kappa_0}{\kappa'}, \frac{\kappa_1}{\kappa'}\right\}\|\nabla T_0(x)-\nabla T_1(x)\|_{HS}^2
\end{align*}
for $m$-a.e. $x\in \R^n$. 
Therefore combining this inequality with (\ref{bb}), we have 
\begin{align}\label{bbb}
&(1-t)\Ent_m(\mu_0) + t\Ent_m(\mu_1) - \Ent_m(\mu_t) \notag \\
&\geq \frac{\kappa}{2}t(1-t)W_2^2(\mu_0, \mu_1) + \frac{1}{2}t(1-t)\min\left\{\frac{\kappa_0}{\kappa'}, \frac{\kappa_1}{\kappa'}\right\}\int_{\R^n}\|\nabla T_0(x)-\nabla T_1(x)\|_{HS}^2 ~dm(x).
\end{align}
It follows from the assumption (i) that $\int_{\R^n}(T_0 - T_1)~dm = \bary(\mu_0) - \bary(\mu_1) = 0$, which implies that by the Poincar\'{e} inequality for $\mu_0$ (Theorem \ref{PI}), 
\[
\kappa \int_{\R^n}(T_{0i}(x) - T_{1i}(x))^2 ~d\mu_0(x) \leq \int_{\R^n}\|\nabla\left(T_{0i}(x) - T_{1i}(x)\right)\|_2^2~d\mu_0(x)
\]
for any $i = 1,2, \dots, n$, where $T_{ji}$ is the $i$-th component of $T_j$ for $j=0,1$.  
Summing up in $i$,  
we have 
\begin{align*}
\int_{\R^n}\|\nabla\left(T_0(x) - T_1(x)\right)\|_{HS}^2~dm(x) &\geq \kappa\int_{\R^n}\|T_0(x) - T_1(x)\|_2^2~dm(x) \\
                                              &\geq \kappa W_2^2(\mu_0, \mu_1).
\end{align*}
Substituting this inequality to (\ref{bbb}), we  finally have 
\[
(1-t)\Ent_m(\mu_0) + t\Ent_m(\mu_1) - \Ent_m(\mu_t) 
\geq \frac{\kappa}{2}\left(1 + \min\left\{\frac{\kappa_0}{\kappa'}, \frac{\kappa_1}{\kappa'}\right\}\right)t(1-t)W_2^2(\mu_0, \mu_1),
\]
which yields the claim. 
\end{proof}

Corollary \ref{MC2} is proved in the same way as Corollary \ref{MC1}. 
Moreover, we can prove some refined functional inequalities as in Corollary \ref{C13} as well. 
\begin{Cor}\label{C21}
Let $m = e^{-V}\mathcal{L}^n$ be a probability measure on $\R^n$ such that $V \in C^{\infty}(\R^n)$ satisfies  $\kappa \leq \nabla^2 V\leq \kappa'$ for some $\kappa, \kappa'>0$, 
and $\mu=e^{-W}\Lc^n\in P_2(\R^n)$ with $\Ent_m(\mu)<\infty$,
where $W$ is a smooth function on $\R^n$. Moreover, we assume the following two conditions: 
\begin{description}
\item{(i)} $\bary(m)=\bary(\mu)$.
\item{(ii)} $\nabla^2 W \geq \tau$ holds for some $\tau>0$. 
\end{description}
Then we have 
\begin{description}
\item{(1)} (HWI inequality) 
\[
\Ent_m(\mu) \leq W_2(\mu, m)\sqrt{I_m(\mu)} - \frac{\kappa}{2}\left(1 + \min\left\{\frac{\kappa}{\kappa'},\frac{\tau}{\kappa'} \right\} \right)W_2^2(\mu,m).
\]
\item{(2)} (logarithmic Sobolev inequality)
\[
\Ent_m(\mu) \leq \frac{1}{2\kappa}\cdot\frac{1}{1 + \min\left\{\frac{\kappa}{\kappa'},\frac{\tau}{\kappa'} \right\}} I_m(\mu).
\]
\end{description}
\end{Cor}

Moreover, (2) yields the following Poincar\'{e} type inequality.
\begin{Cor}\label{C22}
Let $m = e^{-V}\mathcal{L}^n$ be a probability measure on $\R^n$ such that $V \in C^{\infty}(\R^n)$ satisfies  $\kappa \leq \nabla^2 V\leq \kappa'$ for some $\kappa, \kappa'>0$, 
and $f \in H^1(\R^n, m) \cap C^{\infty}(\R^n)$ with bounded second derivatives. 
Moreover, we assume the following two conditions: 
\begin{description}
\item{(i)} $\int_{\R^n}f ~dm =0$.
\item{(ii)} $\int_{\R^n}xf(x) ~dm(x) = 0$.
\end{description}
Then we have 
\[
\kappa\left(1+\frac{\kappa}{\kappa'}\right)\int_{\R^n} f^2 ~dm \leq \int_{\R^n} \|\nabla f\|_2^2 ~dm.
\]
In particular, when $m=\gamma_n$, then for any $f\in H^1(\R^n, \gamma_n) \cap C^{\infty}(\R^n)$ with bounded second derivatives, $\int_{\R^n}f ~d\gamma_n =0$ 
and $\int_{\R^n}\nabla f ~d\gamma_n =0$, we have 
\[
2\int_{\R^n} f^2 ~d\gamma_n \leq \int_{\R^n} \|\nabla f\|_2^2 ~d\gamma_n.
\]
\end{Cor}

\begin{proof}
The proof of this claim is the same as Theorem \ref{PI} with a little modification. Let $\varepsilon>0$ be small enough (with respect to $\delta$ later), and set $\rho:=1+\varepsilon f$ and $\mu:=\rho m$. 
Note that $\mu\in P_2(\R^n)$. 
By the assumption (ii), we have $\bary(\mu)=\int_{\R^n}x(1+\varepsilon f) ~dm = \bary(m)$. 
Take $\delta>0$. 
Since $f$ has bounded second derivatives, setting $W:=V - \log\rho$, we may assume that $\nabla^2 W\geq \kappa - \delta$. 
Since $\mu=e^{-W}\Lc^n$, the same argument as Theorem \ref{PI} with Corollary \ref{C21}(2) yields that 
\[
\kappa\left(1+\frac{\kappa-\delta}{\kappa'}\right)\int_{\R^n}f^2 ~dm \leq \int_{\R^n} \|\nabla f\|_2^2 ~dm.
\]
Since $\delta>0$ was arbitrary, letting $\delta\to+0$, we obtain the first assertion. 

The second claim follows from $\int_{\R^n}xf(x) ~d\gamma_n(x)=\int_{\R^n}\nabla f ~d\gamma_n$ by the integration by parts, 
$\kappa=\kappa'=1$ and the first claim. 
\end{proof}

\begin{Rem}
In Corollary \ref{C22}, two conditions of $f$ mean that $f$ is perpendicular to all constants and linear functions in $H^2(\R^n,m)$. 
Thus, Corollary \ref{C22} is related to the second eigenvalue problem (see \cite{CFM}). 
\end{Rem}
%%%%%%%%%%%%%%%%%%%%%%%%%%%%%%%%%%%%%%%%%%%%%%%%%%%%%%%%%%%%%%%%%%%%%%%%%%%%%%%%%%%%%%%%%%%%%%%%%%%%%%%%%%%%%%%%%%%%%%%%%%%%%%%%%%%%%%%%%%%%%%%%%%%%%%%%%%%%%%%%%%%%%%%%%%%%%%%%%%%%%%%%%%%%%%%%%%
\subsection{An alternative proof and an improvement of Theorem \ref{STI}}

In this subsection, we give a new proof of Theorem $\ref{STI}$ and its extension on the real line. 
In order to prove Theorem $\ref{STI}$, Fathi used a fact on moment measures (see \cite{CK}, \cite{S} for details) which follows from deep optimal transport theory, and a 
reverse logarithmic Sobolev inequality (see \cite{CFGLSW}) which follows from convex geometry (more exactly, the Blaschke-Santal\'{o} inequality). 
In the present paper, we only use tools from optimal transport theory, and as its applications, we give generalizations of well-known facts in convex geometry 
(including the Blaschke-Santal\'{o} inequality) in the next section.

Firstly, we give the following lemma which plays an important role in proving the subsequent result. 
\begin{Lem}\label{STI1}
Let $m = e^{-V}\mathcal{L}^1$ be a probability measure on $\R$ such that $V \in C^{\infty}(\R)$ is $\kappa$-convex for some $\kappa>0$, and 
$f,g\in H^1(\R,m)$ be strictly monotone increasing functions. 
\begin{description}
\item{(1)} If $f$ is odd, and $V$ is even, then it holds that 
\begin{align*}
&-\int_{\R}\log (f'g') ~dm + \kappa\int_{\R} fg ~dm + \int_{\R}(f(x) + g(x))(V' (x) - \kappa x) ~dm(x)\notag\\
&\hspace{8cm}  -2 \int_{\R} xV'(x) ~dm(x) +\kappa\int_{\R} x^2 ~dm(x) \geq 0.
\end{align*}
%また$m$が中心$0$，分散$1/\kappa$の正規分布ではないとき，等号成立は，ある定数$a\in \R$が存在して
%$\R$上で$f(x) =x$かつ$g(x) = x + a$となるときに限る
%（$m$が中心$0$，分散$1/\kappa$の正規分布に従うときは$(2)$を参照）． 
\item{(2)} If $f$ satisfies $\int_{\R}f ~d\gamma_1 =0$, then letting $m = \gamma_1$, we have 
\[
-\int_{\R}\log(f'g')~d\gamma_1 + \int_{\R}fg ~d\gamma_1 - 1\geq 0.
\]
Moreover, the equality holds in (2) if and only if there exist some $a>0$ and $b\in \R$ such that $f(x) = ax$ and $g(x) = x/a + b$ hold on $\R$. 
\end{description}
\end{Lem}

\begin{proof}
(1) Note that since $f$ is an odd and strictly monotone increasing function, $f(x)=0$ is equivalent to $x=0$. 
Set $\alpha := g(0)$. Since $f$ is odd, and $V$ is even, 
we have $\int_{\R}f ~dm = 0$ and $\int_{\R}(V'(x) - \kappa x) ~dm(x) = 0$. Hence, we obtain 
\begin{align}\label{d}
&-\int_{\R}\log (f'g') ~dm + \kappa\int_{\R} fg ~dm + \int_{\R}(f(x) + g(x))(V' (x) - \kappa x) ~dm(x)\notag \\
&\hspace{1cm}  -2 \int_{\R} xV'(x) ~dm(x) +\kappa\int_{\R} x^2 ~dm(x)  \notag \\
&= -\int_{\R}\log (f'g') ~dm + \kappa\int_{\R} f(g - \alpha) ~dm + \int_{\R}(f(x) + g(x) - \alpha)(V' (x) - \kappa x) ~dm(x) \notag \\
&\hspace{1cm}  -2 \int_{\R} xV'(x) ~dm(x) +\kappa\int_{\R} x^2 ~dm(x) .
\end{align}
Since $f,g$ are strictly monotone increasing functions, and $f$ is odd, it follows from the definition of $\alpha$ that $f(g-\alpha)$ is nonnegative on $\R$. 
Thus there exists the nonnegative function $h$ on $\R$ satisfying $h^2 = f(g-\alpha)$. 
Then it holds that $2hh' = f'(g-\alpha) + fg'$ on $\R$. 
This equality and the arithmetic-geometric mean inequality yield that 
\begin{align*}
(h')^2 &= \frac{(f'(g-\alpha))^2 + (fg')^2 + 2f(g-\alpha)f'g'}{4h^2} \\
       &\geq \frac{f(g-\alpha)f'g'}{h^2},
\end{align*} 
which implies that $(h')^2 \geq f'g'$ by the definition of $h$. 
Applying this inequality and the arithmetic-geometric mean inequality again, we obtain an estimate of $(\ref{d})$ from below  such that 
\begin{align*}
&-\int_{\R}\log (f'g') ~dm + \kappa\int_{\R} f(g - \alpha) ~dm  + \int_{\R}(f(x) + g(x) - \alpha)(V' (x) - \kappa x) ~dm(x) \notag \\
&\hspace{1cm}  -2 \int_{\R} xV'(x) ~dm(x) +\kappa\int_{\R} x^2 ~dm(x) \\
&\geq -\int_{\R}\log (h')^2 ~dm + \kappa\int_{\R} h^2 ~dm + 2\int_{\R}|h(x)(V' (x) - \kappa x)| ~dm(x) -2 +\kappa\int_{\R} x^2 ~dm(x) \\~
&= -2\int_{-\infty}^0\log (-h')~dm + \kappa\int_{-\infty}^0h^2 ~dm - 2\int_{-\infty}^0h(x)(V'(x) - \kappa x)~dm(x) + \kappa\int_{-\infty}^0x^2~dm(x)\\
&\hspace{1cm} -2\int_0^{\infty}\log h' ~dm + \kappa\int_0^{\infty}h^2 ~dm + 2\int_0^{\infty}h(x)(V'(x) - \kappa x) ~dm(x) + \kappa\int_0^{\infty}x^2~dm(x) -2 \\
&= 2\int_{-\infty}^0(-h' -1 - \log (-h'))~dm + \kappa\int_{-\infty}^0(h(x) +  x)^2 ~dm(x) \\
&\hspace{3cm} +2\int_0^{\infty}(h' -1 -\log (h')) ~dm + \kappa\int_0^{\infty}(h(x)- x)^2~dm(x). 
\end{align*}
Here, in the first inequality, we used the integration by parts for the fourth term, which yields $\int_{\R}xV'(x) ~dm(x)$ $= 1$. 
The first equality also follows from $V'(x) \leq \kappa x$ on $(-\infty,0]$ and $V'(x) \geq \kappa x$ on $[0,\infty)$, and 
the second equality follows from the integration by parts which yields that $\int_{-\infty}^0h' ~dm = \int_{-\infty}^0hV' ~dm$ and $\int_0^{\infty}h' ~dm = \int_0^{\infty}hV' ~dm$. 
Since the function $x-1-\log x$ on $(0,\infty)$ is nonnegative, it yields the claim. 

%$m$は中心$0$，分散$1/\kappa$の正規分布ではないので，$\R$上で$V'(x) = \kappa x$ではないことに注意すれば，
%等号成立は$\R$上で$f = g -\alpha$かつ$h(x) = |x|$が成立することと同値であり，
%これは$f(x) =x$かつ$g(x) = x + \alpha$であることを意味している．

(2) The method of the proof is the same as (1), but we need a trick since we do not have $f(0)=0$. 
Set $\xi := \sup\{x\in \R~|~f(x)\leq 0\}$ and $\alpha :=g(\xi)$. 
Note that $\xi$ and $\alpha$ are well-defined since $\int_{\R}f ~d\gamma_1 = 0$. 
By the strict monotonicity of $f,g$ and the definition of $\alpha$, 
it holds that $f(g-\alpha)$ is nonnegative. Let $h$ be the nonnegative function on $\R$ satisfying $h^2 = f(g - \alpha)$. 
Then, as in $(1)$, we have $f'g' \leq (h')^2$, which yields that by $\int_{\R}f ~d\gamma_1 = 0$, 
\begin{align*}
&-\int_{\R}\log(f'g') ~d\gamma_1 + \int_{\R}fg ~d\gamma_1 - 1 \\
&= -\int_{\R}\log(f'g')~d\gamma_1 + \int_{\R}f(g-\alpha) ~d\gamma_1 - 1 \\
&\geq -\int_{\R}\log (h')^2~d\gamma_1 + \int_{\R}h^2~d\gamma_1 - 1 \\
&= 2\int_{-\infty}^{\xi}(-h' -1 -\log(-h'))~d\gamma_1 +2\int_{\xi}^{\infty}(h' -1 -\log h') ~d\gamma_1 + 2\int_{-\infty}^{\xi}h' ~d\gamma_1 \\
&\hspace{2cm} -2 \int_{\xi}^{\infty}h' ~d\gamma_1  + \int_{\R}h^2d\gamma_1 +1 \\ 
&= 2\int_{-\infty}^{\xi}(-h' -1 -\log(-h'))~d\gamma_1 +2\int_{\xi}^{\infty}(h' -1 -\log h')~d\gamma_1 + \int_{-\infty}^{\xi}(h(x)+x)^2~d\gamma_1(x) \\
&\hspace{2cm} + \int_{\xi}^{\infty}(h(x) - x)^2~d\gamma_1(x).
\end{align*}
Here, the last equality follows from $\int_{\R}x^2 ~d\gamma_1(x) =1$ and the integration by parts which yields that $\int_{-\infty}^{\xi}h' ~d\gamma_1 =\int_{-\infty}^{\xi}xh(x) ~d\gamma_1(x)$ and 
$\int_{\xi}^{\infty}h' ~d\gamma_1 =\int_{\xi}^{\infty}xh(x) ~d\gamma_1(x)$. 
Since the function $x-1-\log x$ on $(0,\infty)$ is nonnegative, we have the desired inequality. 

The equality holds if and only if we have $f'(g-\alpha) = fg'$ and $h(x)^2 = x^2$ on $\R$, which imply that $\xi=0$ and 
there exists some constant $a>0$ such that $f(x) = ax$ and $g(x) = x/a +\alpha$ hold on $\R$. 
\end{proof}

\begin{Rem}\label{rSTI1}
Lemma \ref{STI1}(1) is proved under the condition that $V$ is even, but we only used the facts that 
$\int_{\R}(V'(x) - \kappa x) ~dm(x)=0$, $V'(x) \leq \kappa x$ on $(-\infty,0]$ and $V'(x) \geq \kappa x$ on $[0,\infty)$. 
Note that the first equation is equivalent to $\bary(m) =0$ by the integration by parts. 
In addition, although we assume that $f$ is odd, it suffices to prove for $f$ satisfying $\int_{\R}f ~dm =0$ and $f(0)=0$. 
Hence, we can prove Lemma \ref{STI1}(1) under weaker conditions (see Lemma \ref{STI1'} for a more general discussion in detail).
\end{Rem}

\begin{proof}[Proof of Theorem \ref{STI} on the real line]
(1) Let $T_1$ and $T_2$ be the optimal transport maps from $m$ to $\mu$ and $\nu$, respectively. 
Then by the Monge-Amp\`{e}re equation (Theorem \ref{MA}) and the $\kappa$-convexity of $V$, we have 
\begin{align*}
\Ent_m(\mu) &= - \int_{\R} \log T_1' ~dm + \int_{\R}(V(T_1(x)) -V(x)) ~dm(x) \\
 & \geq  - \int_{\R} \log T_1' ~dm + \int_{\R}V' (x)(T_1(x) - x) ~dm(x) + \frac{\kappa}{2}\int_{\R} (T_1(x) -x)^2 ~dm(x).
\end{align*}
Similarly, it holds that 
\[
\Ent_m(\nu) \geq - \int_{\R} \log T_2' ~dm + \int_{\R}V' (x)(T_2(x) - x) ~dm(x) +  \kappa\int_{\R} (T_2(x) -x)^2/2 ~dm(x).
\]
Therefore, using $W_2^2(\mu,\nu) = \int_{\R}(T_1 - T_2)^2 ~dm$ by Theorem \ref{1dim}, we obtain 
\begin{align}\label{o}
&\Ent_m(\mu) + \Ent_m(\nu) - \frac{\kappa}{2}W_2^2(\mu,\nu) \notag \\
&\geq -\int_{\R}\log (T_1'T_2') ~dm + \int_{\R}V' (x)(T_1(x) - x) ~dm(x) + \int_{\R}V' (x)(T_2(x) - x) ~dm(x) \notag\\
&\hspace{1.5cm}+\kappa\int_{\R}T_1T_2 ~dm- \kappa\int_{\R}T_1(x)x ~dm(x) - \kappa\int_{\R}T_2(x)x ~dm(x) +\kappa\int_{\R} x^2 ~dm(x) \notag\\
&=  -\int_{\R}\log (T_1'T_2') ~dm + \kappa\int_{\R} T_1T_2 ~dm + \int_{\R}(T_1(x) + T_2(x))(V' (x) - \kappa x) ~dm(x) \notag\\
&\hspace{1.5cm} -2 \int_{\R} xV'(x) ~dm(x) +\kappa\int_{\R} x^2 ~dm(x) .
\end{align}
Since $m$ and $\nu$ are symmetric, $T_2$ is odd. Thus, it follows from Theorem \ref{1dim} and Lemma \ref{STI1}(1) that 
the right hand side above is nonnegative, which yields the claim. 

(2) Set $m=\gamma_1$, and let $T_1$ and $T_2$ be the same as in the proof of (1). Then by the same calculations as in the proof of (1), we obtain 
\begin{align*}
\Ent_{\gamma_1}(\mu) + \Ent_{\gamma_1}(\nu) - \frac{1}{2}W_2^2(\mu,\nu) 
= -\int_{\R}\log (T_1'T_2') ~d\gamma_1 + \int_{\R}T_1T_2 ~d\gamma_1 -1.
\end{align*}
Since the barycenter of $\nu$ is the origin, $T_2$ satisfies $\int_{\R}T_2 ~dm = 0$. 
Hence, it follows from Theorem \ref{1dim} and Lemma \ref{STI1}(2) that the right hand side above is nonnegative, which yields the claim. 

The equality holds if and only if, by Lemma \ref{STI1}(2), there exist some constants $a>0$ and $b\in\R$ satisfying $T_1(x)=ax + b$ and $T_2(x)=x/a$ on $\R$, 
which imply the desired result since $\mu={T_1}_{\#}\gamma_1$ and $\nu={T_2}_{\#}\gamma_1$. 
\end{proof}

In the end of this subsection, we describe that Theorem \ref{STI}(1) on the real line can be refined by modifying  the above proof. 
As noted in Remark \ref{rSTI1}, we can prove Lemma \ref{STI1}(1) under weaker conditions for the probability measure $m$ and the function $f$. 
Precisely, we obtain the following. 
\begin{Lem}\label{STI1'}
Let $m = e^{-V}\mathcal{L}^1$ be a probability measure on $\R$ such that $V \in C^{\infty}(\R)$ is $\kappa$-convex for some $\kappa>0$, and set $\xi \in \argmin \{V(x) ~|~ x\in \R\}$.
Let $f,g\in H^1(\R,m)$ be strictly monotone increasing functions, and set $a:=f(\xi)$ and $b:=g(\xi)$. 
Then it holds that 
\begin{align}\label{l}
&-\int_{\R}\log (f'g') ~dm + \kappa\int_{\R} fg ~dm + \int_{\R}(f(x) + g(x))(V'(x) - \kappa x+ \kappa \xi) ~dm(x)\notag\\
&\hspace{2cm} -2 \int_{\R} xV'(x) ~dm(x) +\kappa\int_{\R} (x - \xi)^2 ~dm(x) \notag\\
&\geq \kappa\left( b\int_{\R}f ~dm + a\int_{\R}g ~dm - ab - (a+b)(\bary(m)-\xi) \right).
\end{align}
\end{Lem}

\begin{proof}
For simplicity, set $f_a:=f-a$ and $g_b:=g-b$. 
Then, subtracting the right hand side of (\ref{l}) from the left one, we have 
\begin{align}\label{lll}
&-\int_{\R}\log (f'g') ~dm + \kappa\int_{\R} fg ~dm + \int_{\R}(f(x) + g(x))(V'(x) - \kappa x+ \kappa \xi) ~dm(x) -2 \int_{\R} xV'(x) ~dm(x)\notag\\
&\hspace{1cm}  +\kappa\int_{\R} (x - \xi)^2 ~dm(x)  - \kappa\left(b\int_{\R}f ~dm + a\int_{\R}g ~dm -ab -(a+b)(\bary(m)-\xi)\right) \notag\\
&=-\int_{\R}\log (f_a'g_b') ~dm + \kappa\int_{\R} f_ag_b ~dm + \int_{\R}(f_a(x) + g_b(x))(V'(x) - \kappa x+ \kappa \xi) ~dm(x) \notag\\
&\hspace{1cm}  -2 \int_{\R} xV'(x) ~dm(x) +\kappa\int_{\R} (x - \xi)^2 ~dm(x),
\end{align}
where $\int_{\R}V' ~dm =0$ is used. 
It follows from the definitions of $a$ and $b$ that $f_ag_b\geq 0$. 
In addition, the definition of $\xi$ yields that $V'(x)\leq \kappa x -\kappa \xi$ on $(-\infty,\xi)$ and $V'(x)\geq \kappa x -\kappa \xi$ on $(\xi,\infty)$. 
Hence, the desired inequality is given by estimating (\ref{lll}) from below as in Lemma \ref{STI1}(1).
\end{proof}

When $V$ is even, and $f$ is odd, we obtain $\xi=\bary(m)=\int_{\R} f ~dm = a =0$, which imply that Lemma \ref{STI1'} includes Lemma \ref{STI1}(1) on the real line. 

Using Lemma \ref{STI1'} instead of Lemma \ref{STI1}, 
we can extend the symmetrized Talagrand inequality on the real line. 
In order to describe it, we give some notations. 
Let $m$ be a probability measure on $\R$. When the density of $m$ with respect to the Lebesgue measure has a  unique point attaining its maximum, 
let us denote by $\xi_m$ that point. Moreover, when $\mu$ is a probability measure on $\R$, and $T$ is the optimal transport map from $m$ to $\mu$, set $\alpha_{m,\mu}:=T(\xi_m)$.
\begin{Thm}\label{rSTI1'}
Let $\mu,\nu\in P_2(\R)$, and 
$m = e^{-V}\mathcal{L}^1$ be a probability measure on $\R$ such that $V \in C^{\infty}(\R)$ is $\kappa$-convex for some $\kappa>0$.
Then we have 
\[
\frac{1}{2}W_2^2(\mu, \nu)\leq \frac{1}{\kappa}(\Ent_{m}(\mu)+\Ent_{m}(\nu)) + \Phi(m,\mu,\nu).
\]
Here, the last term of the right hand side above is 
\begin{align*}
\Phi(m,\mu,\nu) &:=-\alpha_{m,\nu}\bary(\mu) - \alpha_{m,\mu}\bary(\nu) + \alpha_{m,\mu}\alpha_{m,\nu} + (\alpha_{m,\mu}+\alpha_{m,\nu})(\bary(m)-\xi_m) \\
&\hspace{1cm}+ \left(\xi_m  - \bary(m) + \frac{\bary(\mu) + \bary(\nu)}{2}\right)^2 - \left(\bary(m) - \frac{\bary(\mu) + \bary(\nu)}{2}\right)^2.
\end{align*}
In particular, when $m=\gamma_1$, setting 
\[\Phi(\mu,\nu):=\Phi(\gamma_1,\mu,\nu)=-\alpha_{\gamma_1,\nu}\bary(\mu) - \alpha_{\gamma_1,\mu}\bary(\nu) + \alpha_{\gamma_1,\mu}\alpha_{\gamma_1,\nu},
\] 
we have 
\[
\frac{1}{2}W_2^2(\mu, \nu)\leq \Ent_{\gamma_1}(\mu)+\Ent_{\gamma_1}(\nu) + \min\{\Phi(\mu,\nu),-\bary(\mu)\bary(\nu)\}.
\]
\end{Thm}

\begin{proof}
Let $S$ and $T$ be the optimal transport maps from $m$ to $\mu$ and $\nu$, respectively. 
Then it follows from (\ref{o}) that 
\begin{align*}
&\Ent_m(\mu) + \Ent_m(\nu) - \frac{\kappa}{2}W_2^2(\mu,\nu)  \\
&\geq -\int_{\R}\log (S'T') ~dm + \kappa\int_{\R} ST ~dm + \int_{\R}(S(x) + T(x))(V'(x) - \kappa x) ~dm(x) \\
&\hspace{1cm} -2 \int_{\R} xV'(x) ~dm(x) +\kappa\int_{\R} x^2 ~dm(x) \\
&=-\int_{\R}\log (S'T') ~dm + \kappa\int_{\R} ST ~dm + \int_{\R}(S(x) + T(x))(V'(x) - \kappa x+ \kappa \xi_m) ~dm(x) \\
&\hspace{1cm} -2 \int_{\R} xV'(x) ~dm(x) +\kappa\int_{\R} (x - \xi_m)^2 ~dm(x) -\kappa\xi_m\int_{\R}(S+T) ~dm +2\kappa\xi_m\bary(m) -\kappa\xi_m^2 .
\end{align*}
Since $\bary(\mu)=\int_{\R} S ~dm$ and $\bary(\nu)=\int_{\R} T ~dm$, it follows from 
\begin{align*}
&-\kappa\xi_m\int_{\R}(S+T) ~dm +2\kappa\xi_m\bary(m) -\kappa\xi_m^2  \\
&=\kappa\xi_m(2\bary(m) - \bary(\mu) - \bary(\nu))-\kappa\xi_m^2 \\
&=-\kappa\left(\xi_m  - \bary(m) + \frac{\bary(\mu) + \bary(\nu)}{2}\right)^2 + \kappa\left(\bary(m) - \frac{\bary(\mu) + \bary(\nu)}{2}\right)^2 
\end{align*}
and (\ref{l}) that $\Ent_m(\mu) + \Ent_m(\nu) - \kappa W_2^2(\mu,\nu)/2$ is estimated from below by
$-\kappa \Phi(m,\mu,\nu)$, which yields the former inequality. 

The latter claim follows from the former one and Corollary \ref{SSTI}. 
\end{proof}

When $m$ and $\nu$ are symmetric probability measures on $\R$, we have $\xi_m=\bary(m)=\bary(\nu)=\alpha_{m,\nu}=0$, which imply that $\Phi(m,\mu,\nu)=0$. 
Hence, Theorem \ref{rSTI1'} includes Theorem \ref{STI}(1) on the real line. 
On the other hand, when $m=\gamma_1$, for instance under $\bary(\nu)=0$, we have $\Phi(\mu,\nu)=-\alpha_{\gamma_1,\nu}\bary(\mu)  +\alpha_{\gamma_1,\mu}\alpha_{\gamma_1,\nu}$ 
which is negative if $\mu$ satisfies some conditions. 
This fact implies that Theorem \ref{rSTI1'} can give a stronger symmetrized Talagrand inequality than Fathi's one. 
Moreover in this case, we notice that $\alpha_{\gamma_1, \mu}$ (and $\alpha_{\gamma_1, \nu}$) is the L\'{e}vy mean (or median) of $\mu$ (and $\nu$), in other words, 
it holds that $\mu(\{x\in \R ~|~ x\leq \alpha_{\gamma_1, \mu} \})\geq1/2$ and $\mu(\{x\in \R ~|~ x\geq \alpha_{\gamma_1, \mu} \})\geq1/2$. 
%%%%%%%%%%%%%%%%%%%%%%%%%%%%%%%%%%%%%%%%%%%%%%%%%%%%%%%%%%%%%%%%%%%%%%%%%%%%%%%%%%%%%%%%%%%%%%%%%%%%%%%%%%%%%%%%%%%%%%%%%%%%%%%%%%%%%%%%%%%%%%%%%%%%%%%%%%%%%%%%%%%%%%%%%%%%%%%%%%%%%%%%%%%%%%%%%
%%%%%%%%%%%%%%%%%%%%%%%%%%%%%%%%%%%%%%%%%%%%%%%%%%%%%%%%%%%%%%%%%%%%%%%%%%%%%%%%%%%%%%%%%%%%%%%%%%%%%%%%%%%%%%%%%%%%%%%%%%%%%%%%%%%%%%%%%%%%%%%%%%%%%%%%%%%%%%%%%%%%%%%%%%%%%%%%%%%%%%%%%%%%%%%%%
%%%%%%%%%%%%%%%%%%%%%%%%%%%%%%%%%%%%%%%%%%%%%%%%%%%%%%%%%%%%%%%%%%%%%%%%%%%%%%%%%%%%%%%%%%%%%%%%%%%%%%%%%%%%%%%%%%%%%%%%%%%%%%%%%%%%%%%%%%%%%%%%%%%%%%%%%%%%%%%%%%%%%%%%%%%%%%%%%%%%%%%%%%%%%%%%%
\section{Applications to convex geometry}

In this section, we describe two applications of Theorem \ref{rSTI1'}: the concentration of measures and the Blaschke-Santal\'{o} inequality. 
It is well-known that the concentration of measures has deep connections with geometric inequalities such as isoperimetric inequalities, and has applications in many fields. 
The Blaschke-Santal\'{o} inequality is an important and classical inequality in convex geometry. 
These are important in the geometric study of (high-dimensional) Banach spaces (see \cite{BGVV}, \cite{AGM}).

%%%%%%%%%%%%%%%%%%%%%%%%%%%%%%%%%%%%%%%%%%%%%%%%%%%%%%%%%%%%%%%%%%%%%%%%%%%%%%%%%%%%%%%%%%%%%%%%%%%%%%%%%%%%%%%%%%%%%%%%%%%%%%%%%%%%%%%%%%%%%%%%%%%%%%%%%%%%%%%%%%%%%%%%%%%%%%%%%%%%%%%%%%%%%%%%%
\subsection{Concentration of measures}

There is the following result on the concentration of measures (\cite{V2}): 
When $m = e^{-V}\mathcal{L}^n$ is a probability measure on $\R^n$ with $V \in C^{\infty}(\R^n)$ being $\kappa$-convex for some $\kappa>0$, and 
$A\subset \R^n$ is a Borel subset with $m(A)>0$, then $1-m(A_r) \leq m(A)^{-1} \exp(-\kappa r^2/4)$ holds for any $r>0$, 
where $A_r:=\{x\in\R^n ~|~\|a-x\|_2 \leq r,~\exists a\in A\}$. 
In this subsection, we give a refined version of the concentration of measures on the real line as an application of Theorem \ref{rSTI1'}. 

For simplicity, when $m$ is a probability measure on $\R$, and $A\subset \R$ is a Borel subset with $m(A)>0$, then 
we set $\mu_{m,A}:=m(A)^{-1}m|_A$ and $\bary_m(A):=\bary(\mu_{m,A})= m(A)^{-1}\int_{A}x ~dm(x)$. 
We also set $d(A,B):=\inf\{|a-b| ~|~a\in A,b\in B\}$ for any subsets $A,B\subset\R$. 

\begin{Thm}\label{COM}
Let $m = e^{-V}\mathcal{L}^1$ be a probability measure on $\R$ such that $V \in C^{\infty}(\R)$ is $\kappa$-convex for some $\kappa>0$. 
We also assume that $V$ satisfies $0\in \argmin \{V(x)~|~x\in \R\}$, and set for any Borel subsets $A,B\subset \R$ with $m(A),m(B)>0$, 
$\alpha_{m,A}:=\alpha_{m,\mu_{m,A}}$, $\alpha_{m,B}:=\alpha_{m,\mu_{m,B}}$ and 
\[
\Phi(m,A,B):=-\alpha_{m,B}\bary_m(A) - \alpha_{m,A}\bary_m(B) + \alpha_{m,A}\alpha_{m,B}+(\alpha_{m,A}+\alpha_{m,B})\bary(m).
\] 
Then it holds that
\[
m(A)m(B) \leq \exp\left(-\frac{\kappa }{2}d(A,B)^2 + \kappa\Phi(m,A,B)\right).
\]
In particular, when $m=\gamma_1$, setting 
\[
\Phi(A,B):=\Phi(\gamma_1,A,B)=-\alpha_{\gamma_1,B}\bary_{\gamma_1}(A) - \alpha_{\gamma_1,A}\bary_{\gamma_1}(B) + \alpha_{\gamma_1,A}\alpha_{\gamma_1,B},
\] 
we have 
\[
\gamma_1(A)\gamma_1(B) \leq \exp\left(-\frac{1}{2}d(A,B)^2 + \min\left\{\Phi(A,B), - \bary_{\gamma_1}(A)\bary_{\gamma_1}(B)\right\}\right).
\]
\end{Thm}

\begin{proof}
Substituting $\xi_m=0$ to the former inequality in Theorem \ref{rSTI1'} yields that 
\begin{align}\label{j}
\frac{1}{2}W_2^2(\mu_{m, A}, \mu_{m, B})\leq \frac{1}{\kappa}(\Ent_{m}(\mu_{m, A})+\Ent_{m}(\mu_{m, B})) + \Phi(m,A,B). 
\end{align}
Now, it follows from the definition of $\mu_{m, A}$ that $\Ent_m(\mu_{m, A})=m(A)^{-1}\int_A\log m(A)^{-1} ~dm = -\log m(A)$. 
Similarly, we obtain $\Ent_m(\mu_{m, B})=-\log m(B)$. 
On the other hand, the definition of the Wasserstein distance implies that $W_2(\mu_{m, A},\mu_{m, B})\geq d(A,B)$. 
Hence by (\ref{j}), we have 
\begin{align*}
\frac{1}{2}d(A,B)^2 \leq -\frac{1}{\kappa}(\log m(A) + \log m(B)) +\Phi(m,A,B),
\end{align*}
which yields the first claim. 

The second claim follows from the latter inequality in Theorem \ref{rSTI1'} with the same proof as above. 
\end{proof}

%%%%%%%%%%%%%%%%%%%%%%%%%%%%%%%%%%%%%%%%%%%%%%%%%%%%%%%%%%%%%%%%%%%%%%%%%%%%%%%%%%%%%%%%%%%%%%%%%%%%%%%%%%%%%%%%%%%%%%%%%%%%%%%%%%%%%%%%%%%%%%%%%%%%%%%%%%%%%%%%%%%%%%%%%%%%%%%%%%%%%%%%%%%%%%%%%
\subsection{The Blaschke-Santal\'{o} inequality}

The classical and geometric Blaschke-Santal\'{o} inequality is as follows: 
for any compact and convex set $K\subset \R^n$ with $0\in\Int(K)$, 
setting $K^{\circ}:=\{x\in \R^n ~|~ \inn{x}{y}\leq 1,~\forall y\in K\}$, we have $\Lc^{n}(K)\Lc^n(K^{\circ})\leq \Lc^n(\mathrm{B}_2^n)^2$ if $\int_{K^{\circ}}x ~dx=0$, 
where $\mathrm{B}_2^n$ is the unit ball in $(\R^n,\|\cdot\|_2)$. 
$K^{\circ}$ coincides with the unit ball of the dual space of $(\R^n,\|\cdot\|_K)$ which is the Banach space whose unit ball is $K$ (where we do not assume the symmetric property of norms). 
This fact naturally implies the relation to the theory of Banach spaces, and in fact it is applied in the local theory of Banach spaces. 

Artstein-Avidan, Klartag and Milman extended this geometric inequality to its functional version (\cite{AKM}), and  
Lehec improved their result with its simpler proof and the condition of its equality (\cite{L}). 
The result of Lehec is as follows: 
For any measurable functions $f,g$ on $\R^n$ satisfying $f(x)+g(y)\leq -\inn{x}{y}$ for any $(x,y)\in\R^n\times\R^n$, 
$\int_{\R^n}e^{f(x)} ~dx < \infty$ and $\int_{\R^n}e^{g(x)} ~dx<\infty$, 
it holds that \[
\int_{\R^n}e^{f(x)} ~dx \int_{\R^n}e^{g(x)} ~dx \leq \left(\int_{\R^n}e^{-\|x\|^2/2} ~dx\right)^2=(2\pi)^n
\] 
if $\int_{\R^n}xe^{f(x)} ~dx=0$. 
When $f(x)=-\|x\|_{K}^2/2$，$g(x)=-\|x\|_{K^{\circ}}^2/2$ where $K\subset \R^n$ is a compact, convex and symmetric subset with $0\in\Int(K)$, 
it is easily checked that Lehec's result yields the classical Blaschke-Santal\'{o} inequality for $K$. 
In \cite{K} and \cite{FM}, 
several extensions of the functional Blaschke-Santal\'{o} inequality are studied. 
In particular, Klartag proved the functional Blaschke-Santal\'{o} inequality, assuming that $f$ and $g$ are even, 
for log-concave probability measures instead of the Lebesgue measure (\cite{K}). 
Moreover, Fathi revealed the dual relation between the functional Blaschke-Santal\'{o} inequality and the symmetrized Talagrand inequality for the standard Gaussian measure (\cite{F}). 
In this subsection, we extend Klartag's result through Fathi's duality on the real line. 

\begin{Thm}\label{A}
Let $m = e^{-V}\mathcal{L}^1$ be a probability measure on $\R$ such that $V \in C^{\infty}(\R)$ is $\kappa$-convex for some $\kappa>0$. 
We assume that $V$ satisfies $0\in \argmin \{V(x)~|~x\in \R\}$. 
Take measurable functions $F,G$ on $\R$ satisfying $0<\int_{\R}e^{\kappa F} ~dm<\infty$, $\int_{\R}x^2e^{\kappa F} ~dm<\infty$, 
$0<\int_{\R}e^{\kappa G} ~dm<\infty$ and $\int_{\R}x^2e^{\kappa G} ~dm<\infty$. 
Set $\alpha_{m,F}:=\alpha_{m,\mu_F}$, $\alpha_{m,G}:=\alpha_{m,\mu_G}$ and 
\[
\Phi(m,F,G):=-\alpha_{m,G}\bary(\mu_F) - \alpha_{m,F}\bary(\mu_G) + \alpha_{m,F}\alpha_{m,G}+(\alpha_{m,F}+\alpha_{m,G})\bary(m),
\] 
where $\mu_F:=e^{\kappa F} m/\int_{\R} e^{\kappa F} ~dm$ and $\mu_G:=e^{\kappa G} m/\int_{\R} e^{\kappa G} ~dm$. 
If $F(x) + G(y) \leq (x-y)^2/2$ holds for any $(x,y)\in \R\times\R$, then we have 
\begin{align}\label{k}
\int_{\R}e^{\kappa F} ~dm\int_{\R}e^{\kappa G} ~dm\leq e^{\kappa\Phi(m,F,G)}.
\end{align}
In particular, when $m=\gamma_1$, 
setting 
\[
\Phi(F,G):=\Phi(\gamma_1,F,G)=-\alpha_{\gamma_1,G}\bary(\mu_F) - \alpha_{\gamma_1,F}\bary(\mu_G) + \alpha_{\gamma_1,F}\alpha_{\gamma_1,G},
\] 
if $F(x) + G(y) \leq (x-y)^2/2$ holds for any $(x,y)\in \R\times\R$, then 
\begin{align}\label{kk}
\int_{\R}e^{F} ~d\gamma_1\int_{\R}e^{G} ~d\gamma_1\leq \exp\left( \min\{\Phi(F,G),-\bary(\mu_F)\bary(\mu_G)\}\right).
\end{align}
\end{Thm}

\begin{proof}
Note that $\mu_F,\mu_G\in P_2(\R)$, and $\int_{\R}F ~d\mu_F<\infty$ and $\int_{\R}G ~d\mu_G<\infty$ since $F(x) + G(y) \leq (x-y)^2/2$ for any $(x,y)\in \R\times\R$. 
Then by the former inequality in Theorem \ref{rSTI1'}, we obtain 
\begin{align}\label{jj}
\frac{1}{2}W_2^2(\mu_F, \mu_G)\leq \frac{1}{\kappa}(\Ent_{m}(\mu_F)+\Ent_{m}(\mu_G)) + \Phi(m,F,G).
\end{align}
Now, it follows from the definition of $\mu_F$ that 
\[
\Ent_m(\mu_F)=\int_{\R}\log\left(\left(\int_{\R}e^{\kappa F} ~dm\right)^{-1}e^{\kappa F}\right) ~d\mu_F =\kappa\int_{\R}F ~d\mu_F - \log \int_{\R}e^{\kappa F} ~dm.
\]
Similarly, we obtain 
\[
\Ent_m(\mu_G) =\kappa\int_{\R}G ~d\mu_G - \log \int_{\R}e^{\kappa G} ~dm.
\]
On the other hand, by the Kantorovich duality (Theorem \ref{KD}), we obtain 
\[
\frac{1}{2}W_2^2(\mu_F, \mu_G) \geq \int_{\R}F ~d\mu_F + \int_{\R}G ~d\mu_G. 
\]
Hence, it follows from (\ref{jj}) that 
\[
\int_{\R}F ~d\mu_F + \int_{\R}G ~d\mu_G \leq \frac{1}{\kappa}\left(\kappa\int_{\R}F ~d\mu_F - \log \int_{\R}e^{\kappa F} ~dm + \kappa\int_{\R}G ~d\mu_G - \log \int_{\R}e^{\kappa G} ~dm\right) +\Phi(m,F,G),
\]
which yields the desired inequality. 

The latter claim follows from the latter inequality in Theorem \ref{rSTI1'} with the same proof as above. 
\end{proof}

When $m$ satisfies $\int_{\R}e^{\kappa x^2/2} ~dm<\infty$, setting $\widetilde{m}:=e^{\kappa x^2/2}m/\int_{\R}e^{\kappa x^2/2} ~dm$, $f(x):=F(x)-x^2/2$ and $g(x):=G(x)-x^2/2$, 
the conditions on $F,G$ are equivalent to $f(x) + g(y) \leq -xy$ for any $(x,y)\in \R\times\R$, and since $\int_{\R}e^{-\kappa x^2/2} ~d\widetilde{m} =\left(\int_{\R}e^{\kappa x^2/2} ~dm\right)^{-1}$, 
$(\ref{k})$ can be represented as 
\[
\int_{\R}e^{\kappa f} ~d\widetilde{m}\int_{\R}e^{\kappa g} ~d\widetilde{m}\leq \left(\int_{\R}e^{-\frac{\kappa}{2}x^2 + \frac{\kappa}{2}\Phi(m,F,G)} ~d\widetilde{m}\right)^2.
\]
Let $\widetilde{V}$ be the function on $\R$ satisfying $\widetilde{m}=e^{-\widetilde{V}}\Lc^1$. Then $\widetilde{V}''\geq 0$ holds, and hence $\widetilde{m}$ is a log-concave probability measure. 
In particular, if $V$ and $F$ (or $V$ and $G$) are even, it follows that $\Phi(m,F,G)=0$, which implies Klartag's result. 

%%%%%%%%%%%%%%%%%%%%%%%%%%%%%%%%%%%%%%%%%%%%%%%%%%%%%%%%%%%%%%%%%%%%%%%%%%%%%%%%%%%%%%%%%%%%%%%%%%%%%%%%%%%%%%%%%%%%%%%%%%%%%%%%%%%%%%%%%%%%%%%%%%%%%%%%%%%%%%%%%%%%%%%%%%%%%%%%%%%%%%%%%%%%%%%%
%%%%%%%%%%%%%%%%%%%%%%%%%%%%%%%%%%%%%%%%%%%%%%%%%%%%%%%%%%%%%%%%%%%%%%%%%%%%%%%%%%%%%%%%%%%%%%%%%%%%%%%%%%%%%%%%%%%%%%%%%%%%%%%%%%%%%%%%%%%%%%%%%%%%%%%%%%%%%%%%%%%%%%%%%%%%%%%%%%%%%%%%%%%%%%%%
%%%%%%%%%%%%%%%%%%%%%%%%%%%%%%%%%%%%%%%%%%%%%%%%%%%%%%%%%%%%%%%%%%%%%%%%%%%%%%%%%%%%%%%%%%%%%%%%%%%%%%%%%%%%%%%%%%%%%%%%%%%%%%%%%%%%%%%%%%%%%%%%%%%%%%%%%%%%%%%%%%%%%%%%%%%%%%%%%%%%%%%%%%%%%%%%

\end{document}